\documentclass[reqno,12pt]{amsart}

\usepackage{amsmath,amsthm,amssymb,amsfonts,amscd,comment}
\usepackage{xypic}

\theoremstyle{plain} 
\newtheorem{thm}{Theorem}[section]
\newtheorem*{thm*}{Theorem}
\newtheorem{cor}[thm]{Corollary}
\newtheorem{lem}[thm]{Lemma}
\newtheorem{prop}[thm]{Proposition}
\newtheorem{conj}[thm]{Conjecture}
\newtheorem*{conj*}{Conjecture}

\newtheorem{dfn}[thm]{Definition}
	
\theoremstyle{definition}
\newtheorem{eg}[thm]{Example}

\newtheorem{rem}[thm]{Remark}	
\theoremstyle{remark}

\newtheorem*{pf}{Proof}

\numberwithin{equation}{section}
\def\NN{{\mathbb N}}
\def\ZZ{{\mathbb Z}}
\def\QQ{{\mathbb Q}}
\def\RR{{\mathbb R}}
\def\CC{{\mathbb C}}
\def\PP{{\mathbb P}}

\def\LL{{\mathbb L}}
\def\PP{{\mathbb P}}

\def\A{{\mathcal A}}
\def\B{{\mathcal B}}

\def\D{{\mathcal D}}
\def\E{{\mathcal E}}
\def\F{{\mathcal F}}

\def\N{{\mathcal N}}
\def\P{{\mathcal P}}

\def\T{{\mathcal T}}

\newcommand{\Aut}{{\rm Aut}}
\newcommand{\Pic}{{\rm Pic}}
\newcommand{\Hom}{{\rm Hom}}
\newcommand{\RHom}{{\rm RHom}}
\newcommand{\dHom}{{\rm hom}}

\newcommand{\iso}{\xrightarrow{\sim}}

\newcommand{\ch}{\mathrm{ch}}

\usepackage{tikz-cd}
\usepackage{aliascnt}
\usepackage{graphicx}

\begin{document}
\title[Spherical twists, relations and the center]{Spherical twists, relations and the center of autoequivalence groups of K3 surfaces}
\author{Federico Barbacovi}
\address{Department of Mathematics, University College London, 25 Gordon Street, London, WC1H 0AY, UK}
\email{federico.barbacovi.18@ucl.ac.uk}

\author{Kohei Kikuta}
\address{Department of Mathematics, Graduate School of Science, Osaka University, Toyonaka Osaka, 560-0043, Japan}
\email{kikuta@math.sci.osaka-u.ac.jp}

\begin{abstract}
Homological mirror symmetry predicts that there is a relation between autoequivalence groups of derived categories of coherent sheaves on Calabi--Yau varieties, and the symplectic mapping class groups of symplectic manifolds. 
In this paper, as an analogue of Dehn twists for
closed oriented real surfaces, we study spherical twists for dg-enhanced triangulated categories. 

We introduce the intersection number and relate it to group-theoretic properties of spherical twists. 
We show an inequality analogous to a fundamental inequality in the theory of mapping class groups about the behavior of the intersection number via iterations of Dehn twists.
We also classify the subgroups generated by two spherical twists using the intersection number. 
In passing, we prove a structure theorem for finite dimensional dg-modules over the graded dual numbers and use this to describe the autoequivalence group. 

As an application, we compute the center of autoequivalence groups of derived categories of K3 surfaces. 
\end{abstract}
\maketitle
\tableofcontents
\section{Introduction}
Let $X$ be a smooth projective variety over a field $K$ and $\D^b(X)$ the bounded derived categories of coherent sheaves on $X$. 
The autoequivalence group $\Aut(\D^b(X))$ consisting of exact self-equivalences of $\D^b(X)$ is an interesting object in group theory. 
There are some attempts to compute $\Aut(\D^b(X))$, but
this problem is rather difficult in general. 
The aim of this paper is to study group structures of $\Aut(\D^b(X))$ by focusing on spherical twists. 
The details are explained in the following two subsections. 

\subsection{Spherical twists and intersection number}
Homological mirror symmetry predicts that there is a relation between autoequivalence groups of derived categories of coherent sheaves on Calabi--Yau varieties, and the symplectic mapping class groups of symplectic manifolds. 
As an analogue of Dehn twists along Lagrangian spheres, Seidel--Thomas introduced ($d$-)spherical objects and an autoequivalence $T_E\in\Aut(\D^b(X))$ called the spherical twist along a spherical object $E\in\D^b(X)$ (\cite{ST}). 
In the following, we consider a more general triangulated category $\D$ with a dg-enhancement, see Section \ref{section-preliminaries-dg} for detailed settings. 

To clarify the analogy, we denote the sum of dimensions of all extension groups by $i(M,N)$ \emph{i.e.} for $M,N\in\D$,
\[
i(M,N):=\sum_{p\in\ZZ}\dim_K\Hom_\D(M,N[p]), 
\]
which we call the intersection number of $M$ and $N$ in this paper. 
Referring to some well-known facts about Dehn twists for real surfaces (see Section \ref{section-MCG}), we prove the following three theorems. 
The intersection number is essential to understand the statements and their proofs. 

The first result is the inequality describing the behavior of the intersection number via iterations of spherical twists. 

\begin{thm}[Theorem \ref{fund-ineq-k-weak} and cf. Theorem \ref{fund-ineq-Dehn}]\label{intro-fund-ineq-k-weak}
Let $E\in\D$ be a spherical object and $M,N\in\D$ objects.  
For any $k\in\ZZ\backslash\{0\}$, we have
\[
i(E,M)i(E,N)\le
i(T^k_EM,N)+i(M,N).
\]
\end{thm}

In the mapping class groups of real surfaces, 
a mapping class commutes with a Dehn twist along a simple closed curve
if and only if 
it preserves the curve up to isotopy. 
In autoequivalence groups, this basic fact corresponds to the following. 
\begin{thm}[Theorem \ref{commute-sph} and cf. Theorem \ref{Dehn-commute-sph}]
\label{commute-sph-intro}
Let $E_1,E_2\in\D$ be spherical objects, $\Phi\in\Aut(\D)$ an autoequivalence and
$k_1,k_2\in\ZZ\backslash\{0\}$. 
Then the following are equivalent: 
\begin{enumerate}
\item
$\Phi\circ T_{E_1}^{k_1}=T_{E_2}^{k_2}\circ\Phi$. 
\item
$\Phi(E_1)=E_2[l]$ for some $l\in\ZZ$, and $k_1=k_2$. 
\end{enumerate}
\end{thm}

As a corollary, we show a one-to-one correspondence between spherical objects and spherical twists (Corollary \ref{one-to-one-conjugate} (i)). 

It is well-known that 
two Dehn twists whose intersection number is greater than one generate the (non-abelian) free group of rank 2. 
We prove the corresponding result for autoequivalence groups. 
\begin{thm}[Theorem \ref{free-sph-twist} and cf. Theorem \ref{free-Dehn-twist}]\label{intro-free-sph-twist}
Let $E_1,E_2\in\D$ be spherical objects satisfying $E_1\not\simeq E_2[l]$ for any $l\in\ZZ$. 
If $i(E_1,E_2)\ge2$, then $\langle T_{E_1}^{k_1},T_{E_2}^{k_2} \rangle$ is isomorphic to the free group $F_2$ of rank $2$ for any $k_1,k_2\in\ZZ\backslash\{0\}$. 
\end{thm}
We use the ping-pong lemma for the proof. 
This approach is the same as the case of mapping class groups. 
As a corollary, we reveal the relationship between the intersection number and presentations of subgroups generated by two spherical twists (Corollary \ref{braid-rel}, \ref{comm-rel} and \ref{free-rel}). 

\subsection{The center of autoequivalence groups of K3 surfaces}
Let $X$ be a complex algebraic K3 surface and $\Aut_{{\rm CY}}(\D^b(X))$ the subgroup of autoequivalences trivially acting on the transcendental lattice of $X$. 
The autoequivalence groups are usually studied by using the action on the cohomology. 
In contrast to the automorphism groups of K3 surfaces, there are non-trivial autoequivalences trivially acting on the cohomology: squares of spherical twists for example. 
These cohomologically trivial autoequivalences are detected by the action on the space ${\rm Stab}(X)$ of stability conditions, which is Bridgeland’s approach in \cite{BriK3}. 
Then we naturally consider the subgroups $\Aut^{\dagger}(\D^b(X))$ and $\Aut^{\dagger}_{{\rm CY}}(\D^b(X))$ which preserve the distinguished component of ${\rm Stab}(X)$ (see Definition \ref{dist-pres-auteq}). 

The center (Definition \ref{def-center}) measures the commutativity of a given group. 
The triviality of the center of the mapping class group is proved via 
the equivalence between the commutativity with Dehn twists and the fixability of simple closed curves (Theorem \ref{Dehn-commute-sph}). 
Similarly, as an application of Theorem \ref{commute-sph-intro}, 
we compute the center groups $Z(\Aut^{\dagger}(\D^b(X)))$ and $Z(\Aut^{\dagger}_{{\rm CY}}(\D^b(X)))$. 
\begin{thm}[Theorem \ref{main}]\label{main-intro}
Let $X$ be a K3 surface of any Picard rank, $m_X$ the  order of the finite cyclic group 
$\Aut_t(X):=
\left\{f\in\Aut(X)~\middle|~ H^2(f)|_{{\rm NS}(X)}=\mathrm{id}_{{\rm NS}(X)} \right\}$,
and $f_t$ a generator of $\Aut_t(X)$.
Then we have the following
\begin{enumerate}
\item
$Z(\Aut^{\dagger}(\D^b(X)))
=\Aut_t(X)\times\ZZ[1]
\simeq(\ZZ/m_X)\times\ZZ$. 
\item
\begin{eqnarray*}
Z(\Aut^{\dagger}_{{\rm CY}}(\D^b(X)))
&=&
\begin{cases}
\langle ~(f_t^*)^{m_X/2}\circ[1]~\rangle & \mbox{if }m_X\mbox{ is even}\\
\ZZ[2] & \mbox{if }m_X\mbox{ is odd}
\end{cases}\\
&\simeq&\ZZ. 
\end{eqnarray*}
\end{enumerate}
\end{thm}
We also compute the center of the quotient $\Aut^{\dagger}_{{\rm CY}}(\D^b(X))/\ZZ[2]$ (Corollary \ref{orb-fund}), which is closely related to the orbifold fundamental group of the stringy K\"ahler moduli space of $X$. 
These results reveal that the number $m_X$ determines the group structure of the center groups, so we explain some examples of $m_X$ in \S \ref{subsec-eg}. 

\subsection{Related works}

We give some comments on related works with Theorem \ref{intro-fund-ineq-k-weak} and Theorem \ref{intro-free-sph-twist}. 
Historically, in the case of mapping class groups, the proof for the freeness (Theorem \ref{free-Dehn-twist}) of a subgroup generated by two Dehn twists was published by Ishida in 1996 (\cite{Ish}). 
Then ping-pong lemma and the inequality (Theorem \ref{fund-ineq-Dehn}) about the intersection number were key to their proof. 
In a similar manner, 
for the $\ZZ/2$-graded Fukaya ($A_{\infty}$-)categories of exact symplectic manifolds with contact type boundary, 
Keating proved the inequality (\cite[Proposition 7.4]{Kea}) analogous to Theorem \ref{fund-ineq-Dehn}, and the freeness (\cite[Theorem 1.1]{Kea}) for Dehn twists along Lagrangian spheres by ping-pong lemma. 
For algebraic triangulated categories, 
Volkov also proved similar results: the inequality (\cite[Lemma 3.3]{Volk}) under technical assumptions and the freeness (\cite[Theorem 2.7(4)]{Volk}). 
As an another approach to the freeness, in his thesis (\cite[Theorem 4.4]{Kim}), 
Jongmyeong Kim proved that general $n$ spherical twists whose intersection number is greater than one respectively, 
generate the free group of rank $n$ under the formality assumption for some dg-algebra obtained by spherical objects. 
He translates and reformulates Humphries' argument into a categorical setting. 
Ansch\"utz also proved the freeness in a very special case: $\D$ is a CY2-category and $i(E_1, E_2)=2$ in his master thesis (\cite[Theorem 1.2]{Ans}).

In this paper, We use different techniques to prove the inequality (Theorem \ref{intro-fund-ineq-k-weak}) and the freeness (Theorem \ref{intro-free-sph-twist}). 
In our approach expanded in Section \ref{section-preliminaries-dg}~$\sim$~\ref{section-pf-inequality}, we do not need the extra assumptions to prove the inequality. 
The inequality itself is important for some applications. 
Actually, the inequality implies Theorem \ref{commute-sph-intro} and this theorem implies the one-to-one correspondence between the group $\langle T_{E_1},T_{E_2}\rangle$ and intersection number $i(E_1,E_2)$ (Corollary \ref{braid-rel}, \ref{comm-rel} and \ref{free-rel}), and is used in the computations of the center of autoequivalence groups of K3 surfaces in Section \ref{section-center}. 
Moreover, we treat the power of spherical twists throughout this paper. This is important in future applications. 
The formulation of this paper clarifies the analogy between the autoequivalence groups and the mapping class groups. 
We also obtain a description of modules over the graded dual numbers (Proposition \ref{class-fin-dim-mod}) and autoequivalences of its derived categories (Appendix \ref{appendix}). 


\bigskip
\noindent
{\bf Acknowledgements.}
F.B. was supported by the European Research Council (ERC) under the European Union Horizon 2020 research and innovation programme (grant agreement No.725010). 
K.K. is supported by JSPS KAKENHI Grant Number 20K22310 and 21K13780. 

\bigskip
\noindent
{\bf Notation and Convention.}
Let $K$ be a field. 
\begin{itemize}
\item
For a $K$-linear triangulated category $\D$, 
an {\it autoequivalence} of $\D$ is a $K$-linear exact self-equivalence $\D\to\D$.  
The {\it autoequivalence group} $\Aut(\D)$ is the group of (natural isomorphism classes) of autoequivalences of $\D$. 
\item
For a smooth projective variety $X$ over $K$, 
the category $\D^b(X):=\D^b({\rm Coh}(X))$ is the derived category of bounded complexes of coherent sheaves on $X$.
\item
A {\it K3 surfaces} $X$ means a complex algebraic K3 surface \emph{i.e.} a smooth projective surface over the complex number field $\CC$ such that $\omega_X\simeq \mathcal{O}_X$ and $H^1(X,\mathcal{O}_X)=0$. 
Let $\Aut(X)$ be the automorphism group of $X$. 
\end{itemize}

\section{Preliminaries on dg-categories}\label{section-preliminaries-dg}
For the convenience of the reader, and to set up the notation for dg-categories and related notions, we now briefly recall the notions we will need. 
For a more detailed treatment of dg-categories and their derived categories, the reader is referred to \cite[§ 2.1]{AL}. 
\subsection{Dg-categories and dg-modules}
Let us fix a field $K$.
We will write $\textbf{Mod}\text{-}K$ for the category whose objects are couples $(V, d_V)$ where $V = \oplus_{n \in \mathbb{Z}} V_n$ is a $\mathbb{Z}$-graded vector space over $K$ and $d_V \colon V \rightarrow V$ is a $K$-linear endomorphism such that $d_V(V_n) \subset V_{n+1}$ for any $n \in \mathbb{Z}$, and $d_V^2 = 0$. The morphism $d_V$ is called the \emph{differential}.
We will often denote $(V, d_V)$ simply by $V$, and we will call it a $K$-dg-vector space.

Given two $K$-dg-vector spaces $V$ and $W$, morphisms between them are given by
\begin{equation}
    \label{eqn:hom-space-mod-K}
    \Hom_{\textbf{Mod}\text{-}K}(V,W) = \bigoplus_{n \in \mathbb{Z}} \Hom^n_{\textbf{Mod}\text{-}K}(V,W)
\end{equation}
where $f_n \in \Hom^n_{\textbf{Mod}\text{-}}(V,W)$ is a $K$-linear morphism $f \colon V \rightarrow W$ such that $f_n(V_m) \subset W_{m+n}$ for any $m \in \mathbb{Z}$.
The graded $K$-vector space \eqref{eqn:hom-space-mod-K} can be endowed with the differential
\[
    df = d(\{ f_n \}) = \{ d_W \circ f_n - (-1)^n f_n \circ d_V \}
\]
for $f \in \Hom_{\textbf{Mod}\text{-}K}(V,W)$.
Above we introduced the notation $f = \{ f_n \}$, whose right hand side describes the components of $f$ with respect to the direct sum decomposition in \eqref{eqn:hom-space-mod-K}.

The above discussion implies that $\textbf{Mod}\text{-}K$ is naturally enriched over itself.
Even more is true: $\textbf{Mod}\text{-}K$ carries a monoidal structure for which the hom space \eqref{eqn:hom-space-mod-K} is the internal hom.
The tensor product of $V$ and $W$ is defined as
\[
    V \otimes_K W = \bigoplus_{n \in \mathbb{Z}} \left( \bigoplus_{i + j = n} V_n \otimes_K W_m \right)
\]
with differential $d_{V \otimes_K W} = d_V \otimes \mathrm{id} + \mathrm{id} \otimes d_W$.

We can now introduce the notion of a dg-category: a dg-category over $K$ is a small category $\A$ enriched\footnote{We require the composition maps to be closed, degree zero morphisms in $\textbf{Mod}\text{-}K$.} in $\textbf{Mod}\text{-}K$.
Given two dg-categories $\A$ and $\B$, a dg-functor $\Phi \colon \A \rightarrow \B$ is a functor whose induced maps on morphism spaces preserve the degree and the differential.

Notice that to any dg-category $\A$ we can attach a category, called its \emph{homotopy category}, which is denoted by $H^0(\A)$ and whose objects are the same as those of $\A$, but for any $a_1, a_2 \in \A$ we have $\Hom_{H^0(\A)}(a_1, a_2) := H^0(\Hom_{\A}(a_1, a_2))$.

We will be mainly interested in dg-modules over a fixed dg-category $\A$.
An $\A$-dg-module is dg-functor $\A^{\mathrm{op}} \rightarrow \textbf{Mod}\text{-}K$.
Given an $\A$-dg-module $M$, we write $M_a$ for the image of $a \in \A$ via $M$.

We write $\textbf{Mod}\text{-}\A$ for the dg-category of dg-modules over $\A$, and given two $\A$-dg-modules $M$ and $N$ we write $\mathrm{Hom}_{\A}(M,N)$ for the $K$-dg-vector space of morphisms\footnote{The elements of $\mathrm{Hom}_{\A}(M,N)$ are graded natural transformations, see \cite[§ 2.1.1]{AL}.} of $\A$-dg-modules between $M$ and $N$.
Furthermore, we write $\D(\A)$ for the derived category of dg-modules over $\A$, and $\D(\A)^c \subset \D(\A)$ for the subcategory of compact objects.

We write $\P(\A)$ for the full subcategory of $\textbf{Mod}\text{-}\A$ whose objects are given by h-projective dg-modules \emph{i.e.} those modules $P \in \textbf{Mod}\text{-}\A$ such that
\[
    \mathrm{Hom}_{H^0(\textbf{Mod}\text{-}\A)}(P, S) = 0
\]
for any $S \in \textbf{Mod}\text{-}\A$ such that $S_a$ is an acyclic $K$-dg-vector space for any $a \in \A$.

As in the general theory of modules over rings, any dg-module is quasi-isomorphic\footnote{Quasi-isomorphisms of $\A$-dg-modules are defined fiberwise. Namely, a morphism $f \colon M \rightarrow N$ of $\A$-dg-modules is a quasi-isomorphism if $f_a \colon M_a \rightarrow N_a$ is so for any $a \in \A$.} to an h-projective $\A$-dg-module, and we can use h-projective dg-modules both to compute morphisms in the derived category and to derive tensor products of dg-modules, see \cite[§ 2.1.4]{AL} for the definition of the latter.

In the following, we write $\P(\A)^c$ for the subcategory of those h-projective dg-modules that are compact in the derived category.
Furthermore, we write $\otimes=\otimes_K$.

\subsection{Dg-cones}
\label{subsubsection:dg-cone}

The category $\D(\A)$ is a triangulated category with shift functor defined fiberwise \emph{i.e.} given $M \in \D(\A)$ we have $(M[1])_a = M_a[1]$ for any $a \in \A$.

The operation of taking cones in $\D(\A)$ can be strictified to define an operation in $\textbf{Mod}\text{-}\A$.
Namely, given $M, N \in \textbf{Mod}\text{-}\A$ and $f \colon M \rightarrow N$ a closed, degree zero morphism of $\A$-dg-modules, we define a dg-cone of $f$ as a dg-module $C \in \textbf{Mod}\text{-}\A$ endowed with morphisms of degree zero
\[
    \begin{array}{lcr}
        M[1] \xrightarrow{i} C \xrightarrow{p} M[1] & & N \xrightarrow{j} C \xrightarrow{q} N
    \end{array}
\]
such that
\[
    \begin{array}{cccccc}
         pi = \mathrm{id}_{M[1]} & qj = \mathrm{id}_N & ip + jq = \mathrm{id}_{C} & dp = dj = 0 & di = jf & dq = -f p
    \end{array}
\]
It is easy to see that the dg-cone is uniquely defined up to a closed, degree zero isomorphism in $\textbf{Mod}\text{-}\A$, and that for any $f$ the dg-module $C(f) := M[1] \oplus N$ with diffential $d(m,n) = (dm, dn - f(n))$ is a dg-cone of $f$.

\begin{rem}
    \label{rmk:triangle-dg-cone}
    Given $M, N \in \textbf{Mod}\text{-}\A$ and $f \colon M \rightarrow N$ a closed degree zero morphism of $\A$-dg-modules, the dg-cone $C(f)$ is isomorphic, as an element in $\D(\A)$, to the cone of $f$.
    Moreover, we have the distinguished triangle
    \begin{equation}
        \label{eqn:dt-triangle-dg-cone}
        M \xrightarrow{f} N \xrightarrow{j} C(f) \xrightarrow{p} M[1].
    \end{equation}
\end{rem}

\subsection{Twisted complexes}
\label{subsubsection:twisted-complexes}
A notion that will be fundamental for us is that of a (one-sided) twisted complex.
For a more detailed treatment of twisted complexes, the reader is referred to\footnote{A remark is in order. In both the given references, only the notion of a twisted complex with a finite number of non-zero terms is considered. However, the theory still applies for infinite twisted complexes.} \cite{BK} or \cite{AL}.

A one-sided twisted complex over $\A$ is a collection of $\A$-dg-modules $E_i$, $i \in \mathbb{Z}$, and morphisms $\alpha_{ij} \colon E_i \rightarrow E_j$, $i < j$, of degree $i-j+1$ such that
\[
    (-1)^j d\alpha_{ij} + \sum_{i < k < j} \alpha_{kj} \circ \alpha_{ik} = 0
\]
for any $i < j$.

We adopt the following conventions:
\begin{enumerate}
    \item the one-sided twisted complex given by the modules $E_i$ and the morphisms $\alpha_{ij}$ will be denoted by $(E_i, \alpha_{ij})$
    \item \label{item:notation-finite-twisted-complex}
    if $E_i \neq 0$ only for $i = -n, \dots, 0$, and $\alpha_{ij} = 0$ for $j - i > 1$, then we denote the one-sided twisted complex $(E_i, \alpha_{ij})$ by
    \[
        E_{-n} \xrightarrow{\alpha_{-n(-n+1)}} E_{-n+1} \xrightarrow{\alpha_{(-n+1)(-n+2)}} \dots \xrightarrow{\alpha_{-10}} E_0
    \]
\end{enumerate}

One-sided twisted complexes can be packaged into a dg-category.
Given two twisted complexes $(E_i, \alpha_{ij})$ and $(F_i, \beta_{ij})$ morphisms of degree $p$ between them are given by
\begin{equation}
    \label{eqn:hom-space-twisted-complexes}
    \bigoplus_{p = q+l-k} \Hom_{\A}^q(E_k, F_l).
\end{equation}
Given $\gamma \colon (E_i, \alpha_{ij}) \rightarrow (F_i, \beta_{ij})$ a morphism of twisted complexes, the differential of $\gamma$ is defined as follows.
Let us write $\gamma_{q}^{kl} \colon E_k \rightarrow F_l$ for the component $\gamma$ belonging to $\Hom_{\A}^q(E_k, F_l)$.
Then, its differential is given by
\begin{equation}
    \label{eqn:differential-morphisms-twisted-complexes}
    d (\gamma_{q}^{kl}) = (-1)^l d_{\A}(\gamma_q^{kl}) + \sum_{l < m} \beta_{lm} \circ \gamma_q^{kl} - (-1)^{q+l-k} \sum_{m < k} \gamma_q^{kl} \circ \alpha_{mk}
\end{equation}
where $d_{\A}$ is the differential on morphisms in $\A$.

We adopt the following convention to define a morphism of twisted complexes: we write $\gamma = \{ \gamma_q^{kl} \colon E_k \rightarrow F_l\}$ and we mean that $\gamma$ is the morphism of twisted complexes whose component in $\Hom_{\A}^q(E_k, F_l)$ is given by $\gamma_q^{kl}$.
With this convention, if a component is not specified, it means that it is zero.

\subsection{Convolution of twisted complexes}
\label{convolution-twisted-complexes}

Given a one-sided twisted complex $(E_i, \alpha_{ij})$, we define its \emph{convolution} as the $\A$-dg-module $\oplus_{i \in \mathbb{Z}} E_i[-i]$ with differential
\[
    \sum_{i \in \mathbb{Z}} d_{E_i[-i]} + \sum_{i < j} \alpha_{ij}.
\]

The operation of convolution defines a dg-functor from the category of twisted complexes to that of $\A$-dg-modules, see \cite[§ 3.2]{AL}.
In particular, a closed degree zero morphism between twisted complexes induces a closed degree zero morphism between the respective convolutions.

In § \ref{section-pf-inequality} we will convolve various one-sided twisted complexes, and we will need to define elements belonging to such convolutions.
For this reason, we introduce the following notation: if $(E_i, \alpha_{ij})$ is a one-sided twisted complex and $\oplus_{i \in \mathbb{Z}} E_i[-i]$ is its convolution, then we will write
\[
    ( \dots, e_{-3}, e_{-2}, e_{-1}, e_{0}, \dots)
\]
for the element $e \in \oplus_{i \in \mathbb{Z}} E_i[-i]$ whose component in $E_i[-i]$ is given by $e_{i}$, $i \in \mathbb{Z}$.

\subsection{Spherical twists}

From now on we assume that $\A$ is proper \emph{i.e.} that for any $a_1, a_2 \in \A$ the $K$-dg-vector space $\mathrm{Hom}_{\A}(a_1, a_2)$ has finite dimensional total cohomology.
Furthermore, we restrict our attention to those proper dg-categories such that $\D(\A)^c$ has a Serre functor ${\mathcal S}$.
\begin{dfn}\label{def-sph-obj}
For $d\in\mathbb{Z}$, 
an object $E\in\D(\A)^c$ is {\rm $d$-spherical} if it satisfies
$\mathcal{S}(E) \simeq E[d]$ and 
\[
    \bigoplus_{n \in \mathbb{Z}} \mathrm{Hom}_{\D(\A)}(E,E[n])[-n] \simeq k \oplus k[-d].
\]
\end{dfn}
Let $E\in\D(\A)^c$ be a spherical object. 
Then the functor 
\[
-\otimes^{\LL}_K E:~\D(K) \to \D(\A)
\]
is a spherical functor in the sense of \cite[Section 5]{AL}. 
The twist functor $T_E\in\Aut(\D(\A))$ associated to $-\otimes^{\LL}_K E$ is called the {\it spherical twist} along $E$. 
\begin{eg}
When $\D(\A)^c=\D^b(X)$, the first condition in Definition \ref{def-sph-obj} is clearly equivalent to $E\otimes \omega_X\simeq E$. 
For a spherical object $E\in\D^b(X)$, the spherical twist $T_E\in\Aut(\D^b(X))$ along $E$ is isomorphic to a Fourier--Mukai transform $\Phi_{\P_E}$ whose kernel is the cone of the composition of the restriction to the diagonal $\Delta$ with the trace
\[
\P_E:=C\left(E^{\vee}\boxtimes E\rightarrow
(E^{\vee}\boxtimes E)|_{\Delta}\xrightarrow{tr} \mathcal{O}_{\Delta} \right)\in\D^b(X\times X), 
\]
see \cite[Section 8.1]{HuyFM} for details. 
\end{eg}
For a spherical object $E\in\D(\A)^c$ and any object $M\in\D(\A)$, 
the object $T_E(M)$ fits into an exact triangle
\begin{equation}\label{sph-tri}
\Hom^{\bullet}(E,M)\otimes_K E\xrightarrow{ev}M\rightarrow T_E(M), 
\end{equation}
where $\Hom^{\bullet}(M,N):=\bigoplus_{p\in\ZZ}\Hom^p(M,N)[-p]$ and $\mathrm{Hom}^p(M,N):=\mathrm{Hom}_{\D(\A)}(M,N[p])$. 

The following is well-known. 
\begin{lem}[cf. {\cite[Chapter 8]{HuyFM}}]\label{sph-elementary}
Let $E\in\D(\A)^c$ be a spherical object. 
We have the following. 
\begin{enumerate}
\item
$T_E(E)=E[1-d]$. 
\item
$\Phi\circ T_E\circ\Phi^{-1}=T_{\Phi(E)}$
for any autoequivalence $\Phi\in\Aut(\D(\A)^c)$. 
\end{enumerate}
\end{lem}
We note that, if $d\neq1$, any spherical twist is of infinite order by  Lemma \ref{sph-elementary} (i).  

\section{Mapping class groups and Dehn twists}\label{section-MCG}
Dehn twists are fundamental elements in mapping class groups. In this section, 
we recall some well-known facts on Dehn twists, whose analogous statements are considered in the subsequent sections. 

Let $\Sigma_g$ be a connected oriented closed surface of genus $g\ge2$, and ${\rm MCG}(\Sigma_g)$ be the mapping class groups of $\Sigma_g$. 
For an isotopy class of a simple closed curve in $\Sigma_g$, 
the Dehn twist along $a$ is denoted by $T_a\in{\rm MCG}(\Sigma_g)$. 

The following is a fundamental inequality about the behavior of the intersection number via iterations of Dehn twists. 
\begin{thm}[cf. {\cite[Proposition 3.4]{FM}}]\label{fund-ineq-Dehn}
Let $a,b,c$ be isotopy classes of simple closed curves in $\Sigma_g$. 
Then for any $k\in\ZZ\backslash\{0\}$, 
\begin{equation}\label{Dehn-fund-ineq-k-weak}
i(a,b)i(a,c)\le i(T_a^k(c),b)+i(b,c). 
\end{equation}
\end{thm}
This inequality implies the following.  
\begin{thm}[cf. {\cite[\S 3.3 and Fact 3.8]{FM}}]\label{Dehn-commute-sph}
Let $a_1,a_2$ be isotopy classes of simple closed curves in $\Sigma_g$, 
$\phi\in{\rm MCG}(\Sigma_g)$ a mapping class and
$k_1,k_2\in\ZZ\backslash\{0\}$. 
Then the following are equivalent: 
\begin{enumerate}
\item
$\phi\circ T_{a_1}^{k_1}=T_{a_2}^{k_2}\circ\phi$. 
\item
$\phi(a_1)=a_2$ and $k_1=k_2$. 
\end{enumerate}
\end{thm}
\begin{thm}[{\cite[Theorem 1.2]{Ish} and cf. \cite[Theorem 3.14]{FM}}]\label{free-Dehn-twist}
Let $a_1,a_2$ be two isotopy classes of simple closed curves in $\Sigma_g$. 
If the intersection number $i(a_1,a_2)\ge2$, then $\langle T_{a_1}^{k_1}, T_{a_2}^{k_2} \rangle\simeq F_2$ for any $k_1,k_2\in\ZZ\backslash\{0\}$, 
where $F_2$ is the free group of rank $2$. 
\end{thm}

\section{Intersection number}
We consider group-theoretic properties of spherical twists via the intersection number introduced in this section. 

Throughout this section, we suppose $d>1$ and 
$\D:=\D(\A)$ is a derived category of a proper dg-category $\A$ such that
$\D^c:=\D(\A)^c$ has a Serre functor ${\mathcal S}$. 



\begin{dfn}
For any two objects $M,N\in\D$, the
{\rm intersection number} $i(M,N)$ of $M$ and $N$ is defined by
\[
i(M,N):=\sum_{p\in\ZZ}\dHom^p(M,N)\in\ZZ_{\geq0}\sqcup\{+\infty\}, 
\]
where $\dHom^p(M,N):=\dim_K\Hom^p(M,N)$. 
\end{dfn}
For a $d$-spherical object $E\in\D^c$, 
the condition $\mathcal{S}(E)\simeq E[d]$ implies
$i(E,M)=i(M,E)$ for any $M\in\D$. 
\begin{eg}
Let $X$ be a K3 surface and $C_1,C_2$ distinct $(-2)$-curves.
Then
$\mathcal{O}_{C_1}(j_1),\mathcal{O}_{C_2}(j_2)\in\D^b(X)$ are 2-spherical for all $j_1,j_2\in\ZZ$ and 
\[
i(\mathcal{O}_{C_1}(j_1),\mathcal{O}_{C_2}(j_2))
=\dHom^1(\mathcal{O}_{C_1}(j_1),\mathcal{O}_{C_2}(j_2))
=C_1.C_2, 
\]
where $C_1.C_2$ is the intersection number of $C_1$ and $C_2$ on $H^2(X,\ZZ)$. 
\end{eg}
The first main result is a complete analogue of the inequality (\ref{Dehn-fund-ineq-k-weak}) in Theorem \ref{fund-ineq-Dehn} as follows. 
The proof is given in Section \ref{section-pf-inequality}. 
\begin{thm}\label{fund-ineq-k-weak-thm}
    Let $E\in\D^c$ be a $d$-spherical object and $M,N\in\D$ objects such that $i(E,M)$ and $i(E,N)<\infty$. 
    For any $k\in\ZZ\backslash\{0\}$, we have
    \begin{equation}\label{fund-ineq-k-weak}
    i(E,M)i(E,N)\le
    i(T^k_EM,N)+i(M,N).
    \end{equation}
\end{thm}
We then consider the analogue of Theorem \ref{Dehn-commute-sph}. 
\begin{prop}[{\cite[Lemma B.3]{Kim}} and {\cite[Section 4.4]{Kea}}]\label{distinct}
Let $E_1,E_2\in\D^c$ be $d$-spherical objects. 
Then the following are equivalent: 
\begin{enumerate}
\item
There is no integer $l\in\ZZ$ such that $E_1\simeq E_2[l]$.

\item
The composition map
$\Hom^{\bullet}(E_i,E_j)\otimes\Hom^{\bullet}(E_j,E_i)\to \Hom^{\bullet}(E_i,E_i)$
does not hit the identity $\mathrm{id}_{E_i}$ for every $i\neq j$. 
\item
The composition maps
$\Hom^d(E_i,E_i)[-d]\otimes\Hom^{\bullet}(E_i,E_j)\to \Hom^{\bullet+d}(E_i,E_j)$
 and 
$\Hom^{\bullet}(E_i,E_j)\otimes\Hom^d(E_j,E_j)[-d]\to \Hom^{\bullet+d}(E_i,E_j)$
vanish for all $i\neq j$. 
\end{enumerate}
\end{prop}
Two objects $M,N\in\D$ are called {\it distinct} if there is no integer $l\in\ZZ$ such that $M\simeq N[l]$ (cf. Proposition \ref{distinct}(i)). 
\begin{lem}\label{distinct-distinguish}
Let $E_1,E_2\in\D^c$ be distinct $d$-spherical objects. 
Then there exists an object $S\in\D^c$ such that 
\[
i(E_1,S)>i(E_2,S). 
\]
\end{lem}
\begin{pf}
When $i(E_1,E_2)\le1$ (resp. $i(E_1,E_2)\ge3$), it suffices to set $S:=E_1$ (resp. $S:=E_2$). 

Let us consider the case of $i(E_1,E_2)=2$. 
By shifting, we may assume that 
$i(E_1,E_2)=\dHom^0(E_1,E_2)+\dHom^p(E_1,E_2)$ for some $p\ge0$. 
Following the arguments in the proof of \cite[Proposition 5.1]{Kim}, 
we define
\[
Z:=\Hom^d(E_1,E_1)[-d]\otimes E_1\oplus \Hom^{\bullet}(E_2, E_1)\otimes E_2\in\D^c, 
\]
and let $E'_1$ be the cone of the natural evaluation map \emph{i.e.} 
$Z\to E_1\to E'_1$ is an exact triangle in $\D^c$. 
Applying $\Hom(E_j,-)~(j=1,2)$ to this triangle, 
we see that 
$\Hom^d(E_1,Z)\to\Hom^d(E_1,E_1)$ is surjective,  and
$\Hom^i(E_2,Z)\to\Hom^i(E_2,E_1)$ are surjective for all $i\in\ZZ$. 
Moreover by Proposition \ref{distinct} (ii), 
$\Hom^0(E_1,Z)\to\Hom^0(E_1,E_1)$ is zero. 
For $p=0$, we have
\[
\dHom^i(E_1,Z)=
\begin{cases}
5&i=d\\
1&i=2d\\
0&otherwise
\end{cases}
~~\text{  }~~
\dHom^i(E_2,Z)=
\begin{cases}
2&i=d\\
4&i=2d\\
0&otherwise
\end{cases}, 
\]
and for $p>0$, 
\[
\dHom^i(E_1,Z)=
\begin{cases}
1&i=d-p\\
3&i=d\\
1&i=d+p\\
1&i=2d\\
0&otherwise
\end{cases}
~~\text{  }~~
\dHom^i(E_2,Z)=
\begin{cases}
1&i=d-p\\
1&i=d\\
2&i=2d-p\\
2&i=2d\\
0&otherwise
\end{cases}, 
\]
where we have $\dHom^{d+p}(E_1,Z)=2$ (resp. $\dHom^{d}(E_2,Z)=3$) if $d+p=2d$ (resp. $d=2d-p$). 
Direct computations give $i(E_1,E'_1)=6>4=i(E_2,E'_1)$. 
It therefore suffices to set $S:=E'_1$. 
\qed
\end{pf}

The following is the analogue of Theorem \ref{Dehn-commute-sph}. 
\begin{thm}\label{commute-sph}
Let $E_1,E_2\in\D^c$ be $d$-spherical objects, $\Phi\in\Aut(\D^c)$ an autoequivalence and
$k_1,k_2\in\ZZ\backslash\{0\}$. 
Then the following are equivalent: 
\begin{enumerate}
\item
$\Phi\circ T_{E_1}^{k_1}=T_{E_2}^{k_2}\circ\Phi$. 
\item
$\Phi(E_1)=E_2[l]$ for some $l\in\ZZ$, and $k_1=k_2$. 
\end{enumerate}
\end{thm}
\begin{pf}
By Lemma \ref{sph-elementary} (ii), it is enough to show the case $\Phi={\rm id}_{\D^b(X)}$. 
The direction from (ii) to (i) follows from \cite[Proposition 2.6]{ST} and $T_{E[1]}=T_E$. 

We now consider the converse. 
Suppose that $E_1$ and $E_2$ are distinct. 
When $i(E_1,E_2)=0$, 
we have 
\[
T_{E_1}^{k_1}(E_1)=E_1[k_1(1-d)]\neq E_1=T_{E_2}^{k_2}(E_1)
\]
by $d>1$ and $k_1\neq0$, 
hence $T_{E_1}^{k_1}\neq T_{E_2}^{k_2}$. 

When $i(E_1,E_2)=1$, 
there exists $p_0\in\ZZ$ such that $i(E_1,E_2)=\dHom^{p_0}(E_1,E_2)=1$. 
By $T_{E[1]}=T_E$, we may assume that $p_0=1$. 
Then, by \cite[Proposition 5.1 and Remark 5.3]{Kim} and $d>1$, there exists an object $S\in\D$ satisfying $i(E_1,S)=1$ and $i(E_2,S)=0$.
The inequality (\ref{fund-ineq-k-weak}) in Theorem \ref{fund-ineq-k-weak-thm} implies that 
\begin{eqnarray*}
i(T_{E_1}^{k_1}E_2,S)
&\ge&i(E_1,E_2)i(E_1,S)-i(E_2,S)\\
&=&i(E_1,E_2)\\
&>&0
=i(T_{E_2}^{k_2}E_2,S). 
\end{eqnarray*}
We thus have $T_{E_1}^{k_1}\neq T_{E_2}^{k_2}$. 

When $i(E_1,E_2)\ge2$, there exists $S\in\D^c$ satisfying $i(E_1,S)>i(E_2,S)$ by Lemma \ref{distinct-distinguish}. 
The inequality (\ref{fund-ineq-k-weak}) implies that 
\begin{eqnarray*}
i(T_{E_1}^{k_1}E_2,S)
&\ge&i(E_1,E_2)i(E_1,S)-i(E_2,S)\\
&>&i(E_1,E_2)i(E_2,S)-i(E_2,S)=(i(E_1,E_2)-1)i(E_2,S)\\
&\ge&i(E_2,S)=i(T_{E_2}^{k_2}E_2,S). 
\end{eqnarray*}
Hence we have $T_{E_1}^{k_1}\neq T_{E_2}^{k_2}$. 

Finally, since each spherical twist is of infinite order, we have $k_1=k_2$. 
\qed
\end{pf}
We also obtain that powers of spherical twists are not shifts. 
\begin{prop}
Suppose that $\D^c$ has at least two distinct $d$-spherical objects. 
Then each spherical twist $T_E\in\Aut(\D^c)$ satisfies
$T_E^{k_1}\neq [k_2]$ for all $k_1\in\ZZ\backslash\{0\}$ and $k_2\in\ZZ$. 
\end{prop}
\begin{pf}
The claim follows from the same argument of Theorem \ref{commute-sph} and $i(E,M)=i(E[k_2],M)$. 
\qed
\end{pf}
By Theorem \ref{commute-sph} and its proof, we can prove the following. 
\begin{cor}\label{one-to-one-conjugate}
Let $E_1,E_2\in\D^c$ be $d$-spherical objects. 
\begin{enumerate}
\item
$E_1$ and $E_2$ are distinct if and only if $T_{E_1}\neq T_{E_2}$. 
\item
$T_{E_1}^{k_1}$ and $T_{E_2}^{k_2}$ are conjugate in $\Aut(\D^c)$ for some $k_1,k_2\in\ZZ\backslash\{0\}$
if and only if 
$\Phi(E_1)=E_2$ for some $\Phi\in\Aut(\D^c)$. 
\end{enumerate}
\end{cor}

\section{Subgroups generated by two spherical twists}
We consider the analogue of Theorem \ref{free-Dehn-twist}
and relate the intersection number to presentations of subgroups generated by two spherical twists. 

Throughout this section, we suppose $d>1$ and 
$\D:=\D(\A)$ is a derived category of a proper dg-category $\A$ such that
$\D^c:=\D(\A)^c$ has a Serre functor ${\mathcal S}$. 



\begin{lem}[Ping-pong lemma]
Let $G$ be a group acting on a set $W$, and $g_1, g_2$ elements of G. 
Suppose that there are non-empty, disjoint subsets $W_1, W_2$ of $W$ with the property that, for each $i,j(i\neq j)$, we have $g_i^l(W_j)\subset W_i$ for every nonzero integer $l$. 
Then the subgroup $\langle g_1, g_2 \rangle$ generated by $g_1$ and $g_2$ is isomorphic to $F_2$
\end{lem}
Using the ping-pong lemma, we prove the main result in this section. 
\begin{thm}\label{free-sph-twist}
Let $E_1,E_2\in\D^c$ be distinct $d$-spherical objects.
If $i(E_1,E_2)\ge2$, then $\langle T_{E_1}^{k_1},T_{E_2}^{k_2} \rangle\simeq F_2$ for any $k_1,k_2\in\ZZ\backslash\{0\}$. 
\end{thm}
\begin{pf}
To apply the ping-pong lemma, 
we define the subsets $W_1, W_2$ of the set of isomorphism classes of objects in $\D^c$ as follows: 
\begin{eqnarray*}
W_1&:=&\{[S_1]~|~S_1\in\D^c\text{ such that }i(S_1,E_2)>i(S_1,E_1)\}\\
W_2&:=&\{[S_2]~|~S_2\in\D^c\text{ such that }i(S_2,E_1)>i(S_2,E_2)\}. 
\end{eqnarray*}
These are obviously disjoint, and non-empty by Lemma \ref{distinct-distinguish}. 
By the ping-pong lemma, it suffices to check that 
$T_{E_1}^{k_1l}(W_2)\subset W_1$ and $T_{E_2}^{k_2l}(W_1)\subset W_2$ for each $l\in\ZZ\backslash\{0\}$. 
We only show the former inclusion. 

For each $S\in W_2$ and $l\in\ZZ\backslash\{0\}$, 
the inequality (\ref{fund-ineq-k-weak}) in Theorem \ref{fund-ineq-k-weak-thm} gives
\begin{eqnarray*}
i(T_{E_1}^{k_1l}(S),E_2)
&\ge&i(E_1,S)i(E_1,E_2)-i(S,E_2)\\
&\ge&2i(E_1,S)-i(S,E_2)\\
&>&2i(E_1,S)-i(S,E_1)\\
&=&i(E_1,S)=i(E_1,T_{E_1}^{k_1l}(S))=i(T_{E_1}^{k_1l}(S),E_1). 
\end{eqnarray*}
Thus $T_{E_1}^{k_1l}(S)\in W_1$ as desired. 
\qed
\end{pf}
As a corollary, we reveal the relationship between the intersection number and presentations of subgroups generated by two spherical twists. 
\begin{cor}\label{braid-rel}
Let $E_1,E_2\in\D^c$ be distinct $d$-spherical objects.
Then the following are equivalent: 
\begin{enumerate}
\item
$\langle T_{E_1},T_{E_2} \rangle\simeq B_3$, where $B_3$ is the braid group on 3 strands. 
\item
$T_{E_1}T_{E_2}(E_1)=E_2[l]$ for some $l\in\ZZ$
\item
$i(E_1,E_2)=1$.
\end{enumerate}
\end{cor}
\begin{pf}
The assertions ${\rm (iii)}\Rightarrow{\rm (ii)}\Rightarrow{\rm (i)}$ are shown by Seidel--Thomas (\cite[Proposition 2.13 and Theorem 2.17]{ST}) and Nordskova--Volkov (\cite[Theorem 1]{NV}) for more general  $d$-spherical objects, 
see also \cite[Proposition 8.22]{HuyFM}. 
The braid relation gives
$
(T_{E_1}T_{E_2})T_{E_1}(T_{E_1}T_{E_2})^{-1}=T_{E_2}. 
$
Theorem \ref{commute-sph} then implies $T_{E_1}T_{E_2}(E_1)=E_2[l]$ for some $l\in\ZZ$. 

The inequality
\[
\left|i(T_{E_1}E_2,E_2)-i(E_1,E_2)^2\right|\le i(E_2,E_2)=2.
\]
follows from the inequality (\ref{fund-ineq-k-weak}) and easy computations. 
Applying $T_{E_1}T_{E_2}(E_1)=E_2[l]$, 
we have 
\[
i(T_{E_1}E_2,E_2)=i(E_2,T_{E_1}^{-1}E_2)=i(E_2,T_{E_2}^{-1}T_{E_1}^{-1}E_2)=i(E_2,E_1),
\]
hence this inequality holds only in the case of $i(E_1,E_2)=0,1$ or $2$. 
When $i(E_1,E_2)=0$, we have $T_{E_1}=T_{E_2}$ by the braid relation and the commutative relation: $T_{E_1}\circ T_{E_2}=T_{E_2}\circ T_{E_1}$, which contradicts the distinctness.
Assume that $i(E_1,E_2)=2$. 
Then the subgroup $\langle T_{E_1}, T_{E_2}\rangle$ is isomorphic to the rank 2 free group by Theorem \ref{free-sph-twist}, 
which contradicts the braid relation. 
We therefore have $i(E_1,E_2)=1$. 
\qed
\end{pf}
\begin{cor}\label{comm-rel}
Let $E_1,E_2\in\D^c$ be distinct spherical objects.
Then the following are equivalent: 
\begin{enumerate}
\item
$\langle T_{E_1},T_{E_2} \rangle\simeq \ZZ\times\ZZ$
\item
$T_{E_1}(E_2)=E_2[l]$ for some $l\in\ZZ$
\item
$i(E_1,E_2)=0$. 
\end{enumerate}
\end{cor}
\begin{pf}
The assertions ${\rm (iii)}\Rightarrow{\rm (ii)}\Rightarrow{\rm (i)}$ hold
since Lemma \ref{sph-elementary} (ii) implies the commutative relation. 
Clearly, (i) implies (ii) by Theorem \ref{commute-sph}. 

The inequality
\[
\left|i(T_{E_1}E_2,E_2)-i(E_1,E_2)^2\right|\le i(E_2,E_2)=2.
\]
follows from the inequality (\ref{fund-ineq-k-weak}) and easy computations. 
Applying $T_{E_1}(E_2)=E_2[l]$, 
we have 
\[
i(T_{E_1}E_2,E_2)=i(E_2[l], E_2)=2,
\]
hence this inequality holds only in the case of $i(E_1,E_2)=0,1$ or $2$.

When $i(E_1,E_2)=1$, we have $T_{E_1}=T_{E_2}$ by the braid relation (Corollary \ref{braid-rel}) and the commutative relation, which contradicts the distinctness.
Assume that $i(E_1,E_2)=2$. 
Then the subgroup $\langle T_{E_1}, T_{E_2}\rangle$ is isomorphic to the rank 2 free group by Theorem \ref{free-sph-twist}, 
which contradicts the commutative relation. 
We therefore have $i(E_1,E_2)=0$. 
\qed
\end{pf}
\begin{cor}\label{free-rel}
Let $E_1,E_2\in\D^c$ be distinct spherical objects.
Then the following are equivalent: 
\begin{enumerate}
\item
$\langle T_{E_1},T_{E_2} \rangle\simeq F_2$
\item
$i(E_1,E_2)\ge2$. 
\end{enumerate}
\end{cor}
\begin{pf}
The assertion ${\rm (ii)}\Rightarrow{\rm (i)}$ follows from Theorem \ref{free-sph-twist}. 
The converse is given by Corollary \ref{braid-rel} and Corollary \ref{comm-rel}. 
\qed
\end{pf}


\section{Proof of Theorem \ref{fund-ineq-k-weak-thm}}\label{section-pf-inequality}
In this section, we prove Theorem \ref{fund-ineq-k-weak-thm}. 

Throughout this section, 
we consider the case of $d\neq0$, and 
let $\A$ be a proper dg-category such that
$\D(\A)^c$ has a Serre functor ${\mathcal S}$.


\subsection{Good representatives}
The following Proposition \ref{good-rep-sph} is one of the technical steps towards the proof of Theorem \ref{fund-ineq-k-weak-thm} and it is an analogue of \cite[Lemma 3.1]{Kea}.
The difference is that in \cite{Kea} Keating can deform the category so that for a Lagrangian sphere $L$ she has $\Hom(L, L)=K[\epsilon]/\epsilon^2,\ \mathrm{deg}(\epsilon)=d$, $d(\epsilon) = 0$, on the nose. 
We cannot achieve this because we are working in the more strict formalism of dg-categories, but the following proposition will be enough for our purposes.

To simplify the notation, we will write $A_d:=K[\epsilon]/\epsilon^2$ for the dg-algebra of graded dual numbers with $\deg(\epsilon)=d$ and zero differential.

The rationale behind Proposition \ref{good-rep-sph} is that, given a $d$-spherical object $E$, as $d \neq 0$ there exists a unique structure of graded algebra on
\[
     \bigoplus_{n \in \mathbb{Z}} \mathrm{Hom}_{\D(\A)}(E,E[n])[-n].
\]
Namely, we have
\[
     \bigoplus_{n \in \mathbb{Z}} \mathrm{Hom}_{\D(\A)}(E,E[n])[-n] \simeq A_d.
\]
This implies, as $A_d$ is intrinsically formal, see e.g. \cite[Proposition 2.2]{KS-2022}, that the dg-endomorphism-algebra of an h-projective resolution of $E$ is quasi-isomorphic to $A_d$.
What we prove in Proposition \ref{good-rep-sph} is that we can choose the h-projective resolution of $E$ so that it carries and $A_d$-dg-module structure.

\begin{prop}\label{good-rep-sph}
Let $E\in\D(\A)^c$ be a d-spherical object. 
Then, there exists $F\in\P(\A)^c$ with the following properties:
\begin{enumerate}
\item
$F \simeq E$ in $\D(\A)$
\item
there exists $\alpha_{\epsilon} \in\Hom^d_{\P(\A)}(F,F)$ such that $\alpha_{\epsilon} ^2=0$
\item
under the isomorphism of (i), $\alpha_{\epsilon}$ corresponds to the canonical extension $E[-d] \to E$. 
\end{enumerate}
\end{prop}

The proof of the above proposition is quite long and technical, for this reason we split it into various parts.

First, let us fix some notation.
In the following, we write $\varepsilon \colon E[-d] \rightarrow E$ for the canonical extension of $E$ with itself, and we write $E' \in \P(\A)$ for a fixed h-projective resolution of $E$ together with a fixed quasi-isomorphism
\begin{equation}
    \label{eqn:quasi-iso-E}
    E' \rightarrow E.
\end{equation}

We begin by proving the following

\begin{lem}
    \label{lift-system-sph-obj}
    With the notation as above, the sequence of $\A$-dg-modules and morphisms
    \[
        \dots \xrightarrow{\varepsilon} E[-2d] \xrightarrow{(-1)^{d+1}\varepsilon} E[-d] \xrightarrow{\varepsilon} E
    \]
    can be lifted to a one-sided infinite twisted complex in $\P(\A)$.
\end{lem}

\begin{pf}
  The components of the sought twisted complexes will be given by $E'[id]$ for $i \leq 0$.
  Then, to prove the statement of the lemma it is enough to find morphisms
  \[
    \alpha_{ij} \colon E'[id] \rightarrow E'[jd]
  \]
  of degree $i - j + 1$, $i < j$, such that the following diagram commutes for every $i \leq -1$
  \begin{equation}
     \label{eqn:commutative-diagram-lift}
    \begin{tikzcd}
      E'[id] \ar[r, "\alpha_{i(i+1)}"] \ar[d, "\eqref{eqn:quasi-iso-E}"] &[4em] E'[(i+1)d] \ar[d, "\eqref{eqn:quasi-iso-E}"]\\
      E[id] \ar[r, "(-1)^{(i+1)(d+1)}\varepsilon"] & E[(i+1)d]
    \end{tikzcd}
  \end{equation}
  and such that
  \begin{equation}
    \label{eqn:diff-proof-lift}
    (-1)^j d\alpha_{ij} + \sum_{i < k < j} \alpha_{kj} \circ \alpha_{ik} = 0.
  \end{equation}
  Indeed, then the sought twisted complex will be given by $(E'[id], \alpha_{ij})$, $i \leq 0$.
  
  In the following, when we say that $\alpha_{i(i+1)}$ induces the morphism $(-1)^{(i+1)(d+1)}\varepsilon$ via the quasi-isomorphism \eqref{eqn:quasi-iso-E}, we mean that the diagram \eqref{eqn:commutative-diagram-lift} commutes.
  
  We will prove the existence of the morphisms $\alpha_{ij}$ by induction on $j - i$.
  First of all, notice that if we have all the maps $\alpha_{ij}$ for $j - i < n$, then to define the maps $\alpha_{ij}$ with $j-i = n$ it is enough to define $\alpha_{-n0}$.
  Indeed, if we have $\alpha_{-n0}$, then we can set
  \begin{equation}
    \label{eqn:def-ij-from-n0}
    \alpha_{ij} := (-1)^{j(d+1)} \alpha_{-n0}[jd]
  \end{equation}
  for $j-i = n$, and using \eqref{eqn:diff-proof-lift} for $\alpha_{-n0}$ we have\footnote{In the above equations, we make use of the fact that for any morphism $f$ and any $m \in \mathbb{Z}$ it holds that $d(f[m]) = (-1)^m (df)[m]$.}
  \[
    \begin{aligned}
      (-1)^{j}d(\alpha_{ij}) & \, = d(\alpha_{-n 0})[jd] \\
      & \, = - \left(\sum_{-n < k < 0} \alpha_{k0} \circ \alpha_{-n k} \right)[jd]\\
      & \, = - \left(\sum_{i < k+j < j} ((-1)^{j(d+1)}\alpha_{(k+j)j}) \circ ((-1)^{j(d+1)} \alpha_{i (k+j)})  \right)\\
      & \, = - \sum_{i < r < j} \alpha_{rj} \circ \alpha_{ir}.
    \end{aligned}
  \]
  
  We now begin the inductive construction.
  For the case $j - i = 1$ we use that by the definition of an h-projective resolution we have
  \[
    H^0(\Hom_{\A}(E',E'[d])) \simeq \Hom_{\D(\A)}(E,E[d]),
  \]
  and therefore we can define $\alpha_{-10}$ by lifting $\varepsilon$ along the previous isomorphism.
  With this choice and the definition \eqref{eqn:def-ij-from-n0}, we get that $\alpha_{i(i+1)}$ induces $(-1)^{(i+1)(d+1)} \varepsilon$ via the quasi-isomorphism \eqref{eqn:quasi-iso-E}, as we wanted.
  
  Now assume that we defined $\alpha_{ij}$ for any $j-i < n$.
  We claim that to construct $\alpha_{-n0}$ it is enough to prove that $\sum_{-n < k < 0} \alpha_{k0} \circ \alpha_{-nk}$ is a closed morphism.
  Indeed, notice that $\sum_{-n < k < 0} \alpha_{k0} \circ \alpha_{-nk}$ has degree $2-n$ as a morphism $E'[-nd] \rightarrow E'$, and therefore it corresponds to a morphism $E' \rightarrow E'$ of degree $nd+2-n$.
  The cohomology of $\Hom_{\A}(E',E')$ is concentrated in degree $0$ and $d$ because $E$ is $d$-spherical, and therefore a closed element can be non-trivial in cohomology if and only if it has degree $0$ or $d$.
  Now, as we are in the case $n \geq 2$, we have either
  \[
    nd + 2 - n \geq 2(d-1) + 2 = 2d
  \]
  or $nd + 2 - n \leq 2d $, depending on whether $d > 0$ or $d < 0$.
  In either case, the degree of $\sum_{-n < k < 0} \alpha_{k0} \circ \alpha_{-nk}$ is not $0$ or $d$, and therefore if $\sum_{-n < k < 0} \alpha_{k0} \circ \alpha_{-nk}$ is closed it must be the differential of some morphism $E'[-nd] \rightarrow E'$ of degree $1-n$ that we can take to be our $\alpha_{-n0}$.
  
  We now prove that $\sum_{-n < k < 0} \alpha_{k0} \circ \alpha_{-nk}$ is closed.
  We have
  \[
    \begin{aligned}
      & \, d \left( \sum_{-n < k < 0} \alpha_{k0} \circ \alpha_{-nk} \right) \\
      = & \,  \sum_{-n < k < 0} d(\alpha_{k0}) \circ \alpha_{-nk} + (-1)^{k+1} \alpha_{k0} \circ d(\alpha_{-nk})\\
      = & \,  -\sum_{-n < k < 0} \sum_{k < r < 0} \alpha_{0r} \circ \alpha_{kr} \circ \alpha_{-nk} + \sum_{-n < k < 0} (-1)^{2(k+1)} \sum_{-n < s < k} \alpha_{k0} \circ \alpha_{sk} \circ \alpha_{-ns}
    \end{aligned}
  \]
  where in passing from the second to the third line we used that $\alpha_{k0}$ and $\alpha_{-nk}$ satisfy \eqref{eqn:diff-proof-lift} by the induction hypothesis.
  Our aim is to prove that the above terms sum to zero. Take a decomposition of $-n \to 0$ in three steps: $-n < s < r < 0$.
  To get the term $\alpha_{r0} \circ \alpha_{sr} \circ \alpha_{-ns}$ we have two possibilities: either from $d(\alpha_{s0}) \circ \alpha_{-ns}$ or $(-1)^{s+1}\alpha_{r0} \circ d(\alpha_{-ns})$.
  In the first case we get the term $-\alpha_{r0} \circ \alpha_{sr} \circ \alpha_{-ns}$, in the second case $(-1)^{2(s+1)} \alpha_{r0} \circ \alpha_{sr} \circ \alpha_{-ns}$, and they cancel out.

  Hence, given $\alpha_{ij}$ for $j-i < n$, we can define $\alpha_{-n0}$, and the inductive step is complete.
  Thus, the proof of the lemma is complete.
  \qed
\end{pf}

In the following, we write
\begin{equation}
    \label{eqn:twisted-complex-lift}
    (E'[id], \alpha_{ij})
\end{equation}
$i \leq 0$, for the twisted complex constructed in Lemma \ref{lift-system-sph-obj}.

The next step in the proof of Proposition \ref{good-rep-sph} is the following

\begin{lem}
    \label{lem:construction-gamma}
    Let us write $E''$ for the convolution of the twisted complex \eqref{eqn:twisted-complex-lift}.
    Then, there exists a closed morphism $p \colon E'' \rightarrow E''$ of degree $1-d$ such that the dg-cone of $p[-1]$ is $\A$-h-projective.
\end{lem}

\begin{pf}
    We will define a morphism $\pi$ at the level of the twisted complex \eqref{eqn:twisted-complex-lift}, and then convolve it to a morphism $p \colon E'' \rightarrow E''$.
    
    Following the convention introduced in § \ref{subsubsection:twisted-complexes}, we define the morphism of twisted complexes
    \[
        \pi = \left\{  \pi_{-d}^{-j (1-j)} := {\rm id} \colon E'[-jd] \rightarrow E'[(1-j)d] \right\}
    \]
    where $j \geq 1$.
    By definition, $\pi$ has degree $1-d$.
    We now prove that is a closed morphism.
    As per the definition given in § \ref{subsubsection:twisted-complexes}, the differential of $\pi_{-d}^{-k(1-k)}$ is given by
    \[
        \sum_{1-k < m} \alpha_{(1-k) m} \circ \pi_{-d}^{-k (1-k)} - (-1)^{1-d} \sum_{m < -k} \pi_{-d}^{-k (1-k)} \circ \alpha_{m (-k)}.
    \]
    
    To prove that $\pi$ is closed, we fix $j \geq 0$ and we focus our attention to the components of $\pi$ that map to $E'[-jd]$.
    These are given by
    \begin{equation}
        \label{eqn:components-gamma-lem}
        \sum_{1+j < k} \alpha_{(1-k) -j} \circ \pi_{-d}^{-k (1-k)} - (-1)^{1-d} \sum_{m < -1-j} \pi_{-d}^{(-1-j )-j} \circ \alpha_{m (-1-j)}.
    \end{equation}
    Using the definition of $\alpha_{ij}$ given in \eqref{eqn:def-ij-from-n0}, we see that \eqref{eqn:components-gamma-lem} is equal to the shift by $-jd$ of
    \[
    \begin{aligned}
      & \, \left( \sum_{1-k+j < 0} (-1)^{j(d+1)} \alpha_{(1-k+j)0}\circ \pi_{-d}^{(j-k)(1-k+j)} - (-1)^{(j+1)(d+1)}\sum_{m + 1 + j <0 } \pi_{-d}^{-1 0} \circ\alpha_{(m+j)-1} \right)\\
      = & \, (-1)^{j(d+1)} \left( \sum_{n < 0}  \alpha_{n0} \circ \pi_{-d}^{(n-1) n} - (-1)^{(d+1)} \sum_{n < -1} \pi_{-d}^{-1 0} \circ \alpha_{n(-1)} \right) = 0.
    \end{aligned}
    \]
    Here the last equality follows from the fact that, given any $n < 0$, the first term contributes with $\alpha_{n0} \circ \pi_{-d}^{(n-1)n}$, while the second term contributes with $(-1)^{d+1}\pi_{-d}^{-10} \circ \alpha_{(n-1)-1}$.
    Using \eqref{eqn:def-ij-from-n0}, one sees that these two terms cancel out.
    
    Hence, $\pi \colon \eqref{eqn:twisted-complex-lift} \rightarrow \eqref{eqn:twisted-complex-lift}$ is a closed morphism of twisted complexes of degree $1-d$, and we define $p \colon E'' \rightarrow E''[1-d]$ as its convolution.
    
    Let us write $F := C(p[-1])$ for the dg-cone of $p[-1] \colon E''[-1] \rightarrow E''[-d]$ as defined in § \ref{subsubsection:dg-cone}.
    
    We now show that $F$ is $\A$-h-projective\textbf{}.
    It is enough to prove that $E''$ is $\A$-h-projective, and this is what we show.
    The fact that $E''$ is h-projective follows from the fact that it is the convolution of a one-sided twisted complex whose components are $\A$-h-projective, and they are non-zero only in negative degree.
    Indeed, these properties imply that $E''$ has an exhaustive filtration by h-projective $\A$-dg-modules, see also \cite[Tag 09KK]{Sta}, and thus it is $\A$-h-projective.
    \qed
\end{pf}


In the lemma below, we show that $F$ in Lemma \ref{lem:construction-gamma} is an $\A$-h-projective resolution of $E$. 

\begin{lem}
    \label{lem:dg-cone-is-E}
    The dg-cone $F$ of the morphism $p[-1] \colon E''[-1] \rightarrow E''[-d]$ constructed in Lemma \ref{lem:construction-gamma} is quasi-isomorphic to $E$.
\end{lem}

\begin{pf}
    To prove the lemma, we construct a closed, degree zero morphism $f \colon E''[-d] \rightarrow E'$ such that we have a distinguished triangle
    \[
        E''[-d] \xrightarrow{f} E' \rightarrow E'' \xrightarrow{p} E''[1-d]
    \]
    thus proving that $F=C(p[-1]) \simeq E' \simeq E$ in $\D(\A)$.
    
    We define $f$ as the convolution of a closed, degree zero morphism of twisted complexes $\varphi \colon \eqref{eqn:twisted-complex-lift} \rightarrow (E'[d],0)$, where $(E'[d],0)$ is the twisted complex with $E'[d]$ in position zero.
    The degree zero morphism $\varphi$ is defined as
    \[
        \begin{array}{lcr}
            \varphi = \{ \varphi_{n}^{n0} := \alpha_{(n-1)0}[d] \colon E'[-nd] \rightarrow E'[d] \} & & n \geq 0.
        \end{array}
    \]
    
    Let us show that $\varphi$ is closed.
    We fix $j \leq 0$, and we focus our attention on the component of $d\varphi$ starting from $E'[jd]$.
    By definition, this is given by
    \[
        \begin{aligned}
            & d(\varphi^{j0}) - \sum_{j < k \leq 0} \varphi_{0}^{k0} \circ \alpha_{jk} \\
            = \, & (-1)^{d} d(\alpha_{(j-1)0})[d] - \sum_{j < k \leq 0} \alpha_{(k-1)0}[d] \circ \alpha_{jk}\\
            = \, & (-1)^{d} d(\alpha_{(j-1)0})[d] - (-1)^{d+1} \sum_{j-1 < k - 1 < 0} (\alpha_{(k-1)0} \circ \alpha_{(j-1)(k-1)})[d]\\
            = \, & 0
        \end{aligned}
    \]
    where the last equality follows from the fact that the $\alpha_{ij}$ satisfy \eqref{eqn:diff-proof-lift}.
    Hence, $\varphi$ is a closed morphism of degree zero, and upon convolution (and a shift by $-d$) it induces a morphism $f \colon E''[-d] \rightarrow E'$.
    
    To conclude the proof of the lemma, we construct morphisms $i$, $j$, and $q$ such that
    \begin{equation}
        \label{eqn:dg-cone-is-E''}
    \begin{array}{lcr}
      E''[1-d] \xrightarrow{i} E'' \xrightarrow{p} E''[1-d] & \mathrm{and} & E' \xrightarrow{j} E'' \xrightarrow{q} E'
    \end{array}
    \end{equation}
    realise $E''$ as the dg-cone of $f$.
    
    The morphisms $j$ and $q$ are defined as the inclusion of $E'$ into $E''$ and the projection $E''$ onto $E'$, respectively.
    Such maps exists because at the level of underlying graded modules $E''$ is the direct sum $\oplus_{i \leq 0} E'[id]$.
    
    The morphism $i$ is defined as the convolution of the morphism $\iota \colon \eqref{eqn:twisted-complex-lift} \rightarrow \eqref{eqn:twisted-complex-lift}$ defined by
    \[
        \iota = \left\{ \iota_{d}^{(1-j)-j} := \mathrm{id} \colon E'[(1-j)d] \rightarrow E'[-jd] \right\}
    \]
    for $j \geq 1$.
    
    The fact that the morphisms in \eqref{eqn:dg-cone-is-E''} realise $E''$ as the dg-cone of $f$ is an easy check, and we leave it to the reader.
    
    We can now complete the proof by noticing that by Remark \ref{rmk:triangle-dg-cone} we have the distinguished triangle
    \[
        E''[-d] \xrightarrow{f} E' \xrightarrow{j} E'' \xrightarrow{p} E''[1-d]
    \]
    and therefore we have the quasi-isomorphisms
    \[
    F = C(p[-1]) \xleftarrow{j} E' \xrightarrow{\eqref{eqn:quasi-iso-E}} E
    \]
    where we wrote $j \colon E' \rightarrow F$ for the morphism $E' \xrightarrow{j} E'' \rightarrow C(p[-1])$, where the second morphism is the one given by the definition of a dg-cone.
    \qed
\end{pf}

We are now in the position to prove Proposition \ref{good-rep-sph}.

\begin{pf}[Proof of Proposition \ref{good-rep-sph}]

    Lemma \ref{lem:dg-cone-is-E} constructs for us an $\A$-h-projective resolution $F$ of $E$.
    This $F$ will be the one whose existence is claimed in the statement of Proposition \ref{good-rep-sph}.
    
    To conclude the proof of Proposition \ref{good-rep-sph}, we only have to construct the morphism $\alpha_{\epsilon}$.
    We define $\alpha_{\epsilon} \colon F \rightarrow F$ using the direct sum decomposition of $F$.
    Namely, as a graded module $F$ is given by $E'' \oplus E''[-d]$, and when we write $(a,b) \in F$ we mean that $a \in E''$ and $b \in E''[-d]$ in this decomposition.
    Then, $\alpha_{\epsilon}$ is given by
  \[
    F = C(p[-1]) \ni (a,b) \mapsto \alpha_{\epsilon}(a,b) := (0,a) \in C(p[-1]) = F.
  \]
  It is clear that $\alpha_{\epsilon}$ is a closed, degree $d$ morphism such that $\alpha_{\epsilon}^2 = 0$.
  
  We now prove that under the quasi-isomorphisms $F \xleftarrow{j} E' \xrightarrow{\eqref{eqn:quasi-iso-E}} E$ the morphism $\alpha_{\epsilon}$ corresponds to the canonical extension $\varepsilon$ of $E$.
  By the definition of $\alpha_{-10}$, this is equivalent to prove that under the quasi-isomorphism $E' \xrightarrow{j} F$ the morphism $\alpha_{\epsilon}$ corresponds to $\alpha_{-10}$.
  We prove the latter statement.
  
  Take $e \in E'$ a closed element.
  Then (recall the notation for elements belonging to the convolution of a one-sided twisted complex that we introduced in § \ref{convolution-twisted-complexes})
  \[
    \alpha_{\epsilon}(j(e)) = (0, (\dots, 0, 0, e)) = ((\dots, 0, 0, \alpha_{-10}(e)), 0) + d_{F}((\dots, 0, e, 0), 0).
  \]
  Therefore, we get
  \[
    \alpha_{\epsilon}(j(e)) = j(\alpha_{-10}(e)) + d_{F}((\dots, 0, e, 0), 0),
  \]
  which means that $\alpha_{\epsilon} \circ j = j \circ \alpha_{-10}$ in $\D(\A)$, as we wanted.
  Thus, the proof of Proposition \ref{good-rep-sph} is complete.
  \qed
\end{pf}

From now on, every time we have a spherical object $E$ we implicitly assume we replaced it with $F$ as in Proposition \ref{good-rep-sph}. 
Hence, every spherical object $E$ is really $E\in\P(\A)^c$ and there is $\alpha_{\epsilon} \in\Hom^d_{\A}(E,E)$ such that $\alpha_{\epsilon}^2=0$.

Let $E\in\D(\A)^c$ be a $d$-spherical object. The following proposition is an analogue of \cite[Proposition 6.3]{Kea} (recall the convention § \ref{subsubsection:twisted-complexes} \eqref{item:notation-finite-twisted-complex}).

\begin{prop}
For any $N\in\D(\A)$ the object $T^n_E N,\ n\geq0$, is the convolution of the twisted complex
\begin{equation}\label{tot-sph-power}
    \RHom_\A(E,N) \otimes E[-nd] \xrightarrow{d_n} \cdots \xrightarrow{d_2} \RHom_\A(E,N) \otimes E \xrightarrow{ev} N, 
\end{equation}
where $d_k(f\otimes e):=f \alpha_{\epsilon} \otimes e+(-1)^k f \otimes \alpha_{\epsilon} e$. 
\end{prop}

\begin{pf}
    If we had $\RHom_\A(E,E)=\Hom_\A(E,E)\simeq A_d$, we could replicate the proof of \cite[Proposition 6.3]{Kea}. 
    We now explain how to reduce to this case. 
    Notice that the inclusion of the subalgebra $K[\alpha_{\epsilon}] \subset \Hom_\A(E,E)$ is a quasi-isomorphism. 
    Therefore, in the twisted complex representing $T^n_E N$ we can replace $\Hom_\A(E,E)$ with $K[\alpha_{\epsilon}]$ everywhere. 
    The new twisted complex is of the same form as the one of \cite[Proposition 6.3]{Kea}, so we can apply that proof and get the result. 
\qed
\end{pf}

Let us write $C_n(E,N)$ for the convolution of the twisted complex
\begin{equation}
    \label{eqn:c-n-e-n}
    \RHom_\A(E,N) \otimes E[-nd] \xrightarrow{d_n} \cdots \xrightarrow{d_2} \RHom_\A(E,N) \otimes E.
\end{equation}
Then, we have

\begin{prop}\label{good-rep-hom}
    Let $M,N\in\D(\A)$ and $n\geq0$. 
    Then, $\RHom_\A(M,C_n(E,N))$ is isomorphic to the convolution of the twisted complex
    \begin{equation*}
    \RHom_\A(E,N)\otimes \RHom_\A(M,E)[-nd] \xrightarrow{d_n}\cdots\xrightarrow{d_2} \RHom_\A(E,N)\otimes \RHom_\A(M,E), 
    \end{equation*}
    where $d_k(f\otimes g):=f\alpha_{\epsilon}\otimes g+(-1)^k f \otimes \alpha_{\epsilon} g$. 
\end{prop}
\begin{pf}
    The claim follows from the definition of $C_n(E,N)$ and \cite[Lemma 3.4]{AL}.
\qed
\end{pf}

\subsection{Graded dual numbers}
Recall that we write $A_d =K[\epsilon]/\epsilon^2$ for the dg-algebra of the graded dual numbers with $\mathrm{deg}(\epsilon)=d$. 
The aim of this subsection is to classify finite dimensional dg-modules over $A_d$.
We write $\D^b(A_d)$ for the triangulated subcategory of $\D(A_d)$ given by $A_d$-dg-modules with finite dimensional total cohomology.

To prove Proposition \ref{class-fin-dim-mod} below, we will use Koszul duality.
Namely, let us define $K[q]$ as the dg-algebra with $\deg(q) = 1-d$ and $d(q) = 0$.
Then, Koszul duality says that we have an equivalence $\Phi \colon \D(K[q])^c \xrightarrow{\simeq} \D^b(A_d)$ sending $K[q]$ to $K$.
For a recent proof of this statement, see \cite{KS-2022}.

Having $\Phi$ means that to prove a structure theorem for modules in $\D^b(A_d)$ it is enough to study the structure of modules in $\D(K[q])^c$, and this is what we do.

Before moving on to the proposition, let us notice a useful consequence of the equivalence $\D(K[q])^c \simeq \D^b(A_d)$.
By definition, we have $\Hom_{\D(K[q])^c}(K[q][-m(1-d)], K[q]) \simeq K$ for any $m \geq 1$, and the unique non-trivial extension is given by the convolution of the twisted complex
\begin{equation}
    \label{eqn:extension-K[q]}
    K[q][-m(1-d)] \xrightarrow{q^m} K[q],
\end{equation}
which is quasi-isomorphic $K[q]/q^m$.
As $\Phi$ is an equivalence sending $K[q]$ to $K$, we have
\[
    K \simeq \Hom_{\D(K[q])^c}(K[q][-m(1-d)], K[q]) \simeq \Hom_{\D^b(A_d)}(K[-m(1-d)],K),
\]
and therefore there is a unique non-trivial extension of $K$ by itself of degree $m(1-d)$ in $\D^b(A_d)$ for any $m \geq 1$.
This extension is given by the convolution of the twisted complex
\begin{equation}
    \label{eqn:extension-A-d-finite}
    A_d[(1-m)d] \xrightarrow{\epsilon} A_d[(2-m)d] \xrightarrow{\epsilon} \dots \xrightarrow{\epsilon} A_d[-d] \xrightarrow{\epsilon} A_d,
\end{equation}

As $\Phi$ is an equivalence, we get that $\Phi(\eqref{eqn:extension-K[q]})$ is isomorphic, up to a shift that can be computed to be $m(1-d)-1$, to $\eqref{eqn:extension-A-d-finite}$ for any $m \geq 1$.

We write $B_n$, $n \geq 1$, for the convolution of \eqref{eqn:extension-A-d-finite}, and $B_0 = K \in \D^b(A_d)$.
Similarly, we write $C_n$ for the convolution of \eqref{eqn:extension-A-d-finite} with $A_d$ replaced by $A_d^{op}$, and $C_0 = K \in \D^b(A_d^{op})$.
The following proposition gives the desired classification (cf. \cite[Proposition 5.3]{Kea}).

\begin{prop}\label{class-fin-dim-mod}
Let $M$ be a finite dimensional right (resp. left) $A_d$-dg-module.
Then, $M$ is quasi-isomorphic to a finite direct sum of copies of shifts of $B_n$'s (resp. $C_n$'s).
Moreover, $M$ is compact if and only if $B_0$ (resp. $C_0$) does not appear.
\end{prop}

\begin{rem}
    The above statement pairs up with \cite[Proposition 2.2]{KS-2022} to show why $\D(A_d)^{c}(\subsetneq\D^b(A_d))$ is generated by the $A_d$ as a triangulated category (without additional idempotent completion).
\end{rem}

\begin{pf}
We prove the statement for right modules; this will suffice because $A_d$ is commutative.

Let us fix $M \in \D^b(A_d)$ and $M' \in \D(K[q])^c$ such that $\Phi(M') \simeq M$.
Notice that if we forget the grading, then $K[q]$ is a PID.
Hence, the cohomology of the dg-module $M'$ splits as a finite direct sum
\begin{equation}
    \label{eqn:cohomology-split}
H^{\bullet}(M')\simeq K[q][s_1]\oplus\cdots\oplus K[q][s_t]\oplus K[q]/q^{n_1}[m_1]\oplus\cdots\oplus K[q]/q^{n_r}[m_r]
\end{equation}
for some $n_r\ge0,~m_i,s_i\in\ZZ$. 

Let us write $F_{m}$ for the convolution of \eqref{eqn:extension-K[q]}.
Then, notice that $F_m$ is a free resolution of $K[q]/q^{m}$.
Hence, replacing $K[q]/q^{n_i}$ with $F_{n_i}$ in \eqref{eqn:cohomology-split}, we can lift the isomorphism \eqref{eqn:cohomology-split} to a quasi-isomorphism\footnote{In slightly more detail: if $K[q]/q^{n_i}$ appears in \eqref{eqn:cohomology-split}, then it means that there exists $m \in M'$ such that $d(m) = 0$ and $q^{n_i} m = d(m')$ for some $m' \in M$. Then, we define $F_{n_i} \rightarrow M'$ by sending $(1,0), (0,1) \in F_{n_i}$ to $m$ and $m'$, respectively. Here we employed the notation we introduced in § \ref{convolution-twisted-complexes} for elements belonging to the convolution of a twisted complex.}
\begin{equation}
    \label{eqn:dec-M-structure-A-d-mod}
    M'\simeq K[q][s_1]\oplus\cdots\oplus K[q][s_t]\oplus F_{n_1}[m_1]\oplus\cdots\oplus F_{n_r}[m_r]
\end{equation}
Applying $\Phi$ to \eqref{eqn:dec-M-structure-A-d-mod} and using the isomorphisms $\Phi(\eqref{eqn:extension-K[q]}) \simeq \eqref{eqn:extension-A-d-finite}[-m(1-d) + 1]$ and $\Phi(K[q]) \simeq B_0$, we obtain the sought decomposition
\[
    M\simeq\Phi(M')\simeq \bigoplus^t_{i=1}B_0[s_i]\oplus\bigoplus^r_{j=1}B_{n_j}[m_j -n_j(1-d) + 1]. 
\]

Finally, to prove the claim about the compactness of $M$ recall that $\D(A_d)^c$ is closed under taking direct summands and notice that of the $B_n$'s the only non-compact one is $B_0$ (its derived endomorphism algebra is infinite dimensional, see e.g. \cite[Lemma 3.4]{KS-2022}).
\qed
\end{pf}

\subsection{Proof}

We are now ready to prove Theorem \ref{fund-ineq-k-weak-thm}.
We only need one last definition: 

\begin{dfn}
Given $M,N$ two dg-modules over the graded dual numbers $A_d$ and $n\geq0$, we define $(M\otimes_{A_d} N)_n$ as the convolution of the twisted complex
\[
    M \otimes N[-nd] \xrightarrow{d_n} \cdots \xrightarrow{d_3} M \otimes N[-d] \xrightarrow{d_2} M\otimes N, 
\]
where $d_k(m\otimes n):=m\epsilon\otimes n+(-1)^km\otimes \epsilon n$.  
\end{dfn}

\begin{rem}
    \label{rem:independence-tensor-prod}
    The construction $(M \otimes_{A_d} N)_n$ is independent of the quasi-isomorphism classes of $M$ and $N$.
    Indeed, this follows easily from the definition of $(M \otimes_{A_d} N)_n$ and e.g. \cite[Corollary 2.12]{AL-Bar-Cat}.
\end{rem}

We state Theorem \ref{fund-ineq-k-weak-thm} again: 

\begin{thm}
    Let $E\in\D(\A)^c$ be a spherical object and $M,N\in\D(\A)$ objects such that $i(E,M)$ and $i(E,N)<\infty$. 
    For any $k\in\ZZ\backslash\{0\}$, we have
    \begin{equation}\label{fund-ineq-k-weak-appendix}
    i(E,M)i(E,N)\le
    i(T^k_EM,N)+i(M,N).
    \end{equation}
\end{thm}

\begin{pf}
By replacing $M$ by $T_E^{-k}M$, it suffices to prove the claim for $k<0$. 
The case $k=-1$ is proved by direct computations and using $i(\Hom^{\bullet}(E,T_E^{-1}M)\otimes E,N)=i(E,M)i(E,N)$, so we can assume $-k>1$. 

By the definition of $C_{-k}(E,N)$, see \eqref{eqn:c-n-e-n}, we have the following
distinguished triangle
\[
C_{-k}(E,N)\xrightarrow{ev}N\rightarrow T_E^{-k}N. 
\]
Applying $\RHom_\A(M,-)$ to this distinguished triangle, as $i(M,T^{-k}_E N)=i(T^k_E M,N)$ and intersection numbers are subadditive on distinguished triangles, we see that it is enough to prove
\[
\dim_K\RHom_\A(M,C_{-k}(E,N))\ge i(M,E)i(E,N)(=i(E,M)i(E,N)). 
\]
By Proposition \ref{good-rep-hom} we know that
\[
\RHom_\A(M,C_{-k}(E,N))\simeq\left(\RHom_\A(E,N) \otimes_{A_d}\RHom_\A(M,E)\right)_{-k}.
\]
Moreover, by Remark \ref{rem:independence-tensor-prod} we know that we can replace $\RHom_\A(E,N)$ and $\RHom_\A(M,E)$ by quasi-isomorphic $A_d$-dg-modules. 
As $i(E,M)$ and $i(E,N)<\infty$, by Proposition \ref{class-fin-dim-mod} we know that $\RHom_\A(E,N)$ and $\RHom_\A(M,E)$ split as direct sums of shifts of $K$ and $B_r$'s, respectively $C_p$'s, for $r,p\ge1$. 
Hence, to conclude it is enough to prove that
\begin{equation}
    \label{red-ineq}
    \dim_K H^{\bullet}((T_1\otimes_AT_2)_{-k})\ge\dim_K H^{\bullet}(T_1)\cdot\dim_K H^{\bullet}(T_2)
\end{equation}
for $T_1\in\{K,B_r\},~T_2\in\{K,C_p\}$. 

If $T_1=T_2=K$, 
then $(K\otimes_{A_d} K)_{-k}$ is a direct sum of shifts of $K$, and therefore (\ref{red-ineq}) becomes $-k\ge1$, which is true.

If $T_1=K$ and $T_2=C_p$, then the right hand side of (\ref{red-ineq}) is equal to $\dim_K H^{\bullet}(C_p)=2$. 
Hence, we have to find two non-zero cohomology classes in  $(K\otimes_A C_p)_{-k}$.
We claim that the sought classes as given by
\[
    \begin{array}{lcr}
        (0, 0, \dots, 0, 1 \otimes 1) & \mathrm{and} & (1 \otimes (\epsilon, 0, 0, \dots, 0), 0, \dots, 0))
    \end{array}
\]
where we employed the notation we introduced in § \ref{convolution-twisted-complexes} for elements belonging to the convolution of a twisted complex.

It is obvious that $(0, 0, \dots, 0, 1 \otimes 1)$ is closed.
For $(1 \otimes (\epsilon, 0, 0, \dots, 0), 0, \dots, 0)$, it follows from the fact that $1$ and $(\epsilon, 0, 0, \dots, 0)$ are closed in $K$ and $C_p$, respectively, and the definition of the differential in $(K\otimes_A C_p)_{-k}$:
\[
\begin{aligned}
        & d((1 \otimes (\epsilon, 0, 0, \dots, 0), 0, \dots, 0)))\\
        = \; & (d(1 \otimes (\epsilon, 0, 0, \dots, 0)), 1 \cdot \epsilon \otimes (\epsilon, 0, 0, \dots, 0) + (-1)^{k} 1 \otimes (\epsilon^2, 0, 0, \dots, 0), 0, \dots, 0) \\
        = \; & (0, \dots, 0).
\end{aligned}
\]

We prove that $(0, 0, \dots, 0, 1 \otimes 1)$ gives a non-zero cohomology class, the proof for $(1 \otimes (\epsilon, 0, 0, \dots, 0), 0, \dots, 0)$ is analogous.
The element $1 \otimes 1$ is a non-zero cohomology class in $K \otimes C_p$.
Hence, if $(0, 0, \dots, 0, 1 \otimes 1)$ is the differential of some element, it must be the differential of an element of the form $(0, \dots, 0, a_{-1} \otimes b_{-1}, a_0 \otimes b_0)$ with $a_{-1} \otimes b_{-1} \neq 0$.
The differential of such an element is given by (we can assume we always have $1$ on the left because the tensor products are $K$-linear)
\begin{equation}
    \label{eqn:diff-element-to-1-1}
    d((0, \dots, 0, 1 \otimes b_{-1}, 1 \otimes b_0)) = (0, \dots, 0, 1 \otimes d(b_{-1}), 1 \otimes d(b_0) - 1 \otimes \epsilon b_{-1}).
\end{equation}
For \eqref{eqn:diff-element-to-1-1} to be equal to $(0, 0, \dots, 0, 1 \otimes 1)$, we must have
\[
    1 \otimes d(b_0) - 1 \otimes \epsilon b_{-1} = 1 \otimes 1 \iff d(b_0) - \epsilon b_{-1} = 1.
\]
However, by the definition of $C_p$, the equation $d(b_0) - \epsilon b_{-1} = 1$ is not satisfied by any $b_0, b_{-1} \in C_p$, and therefore $1 \otimes 1$ is a non-zero cohomology class.
Hence, we get $\dim_K H^{\bullet}((K\otimes_A C_p)_{-k}) \geq 2$, as we wanted.

If $T_1=B_r$ and $T_2=K$, the situation is as in the previous point. 

If $T_1=B_r$ and $T_2=C_p$, the right hand side of (\ref{red-ineq}) is equal to 4. 
The four elements which give rise to non-zero cohomology classes are given by
\[
    \begin{array}{c}
        (0, \dots, 0, 1 \otimes 1)\\
        ((0, \dots, 0, \epsilon) \otimes (0, \dots, 0, \epsilon), 0, \dots, 0)\\
        ((0, \dots, 0, \epsilon) \otimes (0, \dots, 0, 1), 0, \dots, 0) + (-1)^{k+1} ((0, \dots, 0, 1) \otimes (0, \dots, 0, \epsilon), 0, \dots, 0)\\
        ((\epsilon, 0, \dots, 0) \otimes (\epsilon, 0, \dots, 0), 0, \dots, 0)
    \end{array}
\]

We have exhausted all the possible cases for \eqref{red-ineq}, and therefore we have concluded the proof of the theorem.
\qed
\end{pf}

\begin{rem}
    As mentioned in \S1.3, 
    using different techniques, a particular case of the above theorem was proved by Volkov in \cite[Lemma 3.3]{Volk}.
\end{rem}

%

\section{Preliminaries on K3 surfaces}
In this section, we prepare basic properties of the autoequivalence groups of derived categories of K3 surfaces for the computations of the center groups in Section \ref{section-center}. 

Let $X$ be a K3 surface. 
\subsection{Hodge structures on Mukai lattices}
The integral cohomology group $H^*(X,\ZZ)$ of $X$ has the lattice structure given by the Mukai pairing 
\begin{equation*}\label{Mukai pairing}
((r_1,c_1,s_1), (r_2,c_2,s_2)):=c_1 \cdot c_2-r_1s_2-r_2s_1
\end{equation*}
for $(r_1, c_1, s_1), (r_2,c_2,s_2) \in H^*(X,\ZZ)$.
The lattice $H^*(X,\ZZ)$ called the {\it Mukai lattice} of $X$ is an even unimodular lattice of signature $(4,20)$.
The Mukai lattice has a weight two Hodge structure $\widetilde{H}(X,\ZZ)$ 
given by
$\widetilde{H}(X,\ZZ)\otimes_\ZZ \CC=\bigoplus_{p+q=2}\widetilde{H}^{p,q}(X)$ and
\begin{equation*}
\widetilde{H}^{2,0}(X):=H^{2,0}(X),~
\widetilde{H}^{1,1}(X):=\bigoplus_{p=0}^{2}H^{p,p}(X),~
\widetilde{H}^{0,2}(X):=H^{0,2}(X).
\end{equation*}
This Hodge structure contains the ordinary Hodge structure on $H^2(X,\ZZ)$ as a primitive sub-Hodge structure. 
The algebraic part of $\widetilde{H}(X,\ZZ)$
denoted by $\NN(X)$, 
is equal to $H^0(X,\ZZ)\oplus \mathrm{NS}(X)\oplus H^4(X,\ZZ)$
and has signature $(2, \rho(X))$. 

For an object $E \in \D^b(X)$, the {\it Mukai vector} $v(E) \in H^{2*}(X,\QQ)$
of $E$ is given by
\[
v(E)
:=\ch(E)\sqrt{\mathrm{td}_X}
=(\mathrm{rk}(E), c_1(E), \chi(E)-\mathrm{rk}(E)).
\]
By the Riemann--Roch formula, we have the isomorphism $v:\N(\D^b(X)) \iso \NN(X)$ satisfying
$(v(E),v(F))=-\chi(E,F)$ for any objects $E,F \in \D^b(X)$, 
where $\N(\D^b(X))$ is the numerical Grothendieck group of $\D^b(X)$ and $\chi$ is the Euler pairing on it. 

\subsection{Groups}
A {\it Hodge isometry} $\varphi: \widetilde{H}(X,\ZZ)\to \widetilde{H}(X,\ZZ)$  of $\widetilde{H}(X,\ZZ)$ is an isomorphism of the Hodge structure preserving the Mukai pairing. 
The group of Hodge isometries is denoted by $\mathrm{O}(\widetilde{H}(X,\ZZ))$. 
Let $T(X)$ be the transcendental lattice of $X$ which is the transcendental part of the Hodge structures $\widetilde{H}(X,\ZZ)$ and $H^2(X,\ZZ)$. 
Restricting the Hodge structure $\widetilde{H}(X,\ZZ)$ with the Mukai pairing to sub-Hodge structures $H^2(X,\ZZ),{\rm NS}(X)$ and $T(X)$, 
we similarly define the groups of Hodge isometries $\mathrm{O}(H^2(X,\ZZ)),\mathrm{O}({\rm NS}(X))$ and $\mathrm{O}(T(X))$, respectively. 
Then $\mathrm{O}(T(X))$ is a finite cyclic group, 
and faithfully acts on $H^{2,0}(X)\simeq\CC$ by a root of unity. 

Using the action of $\Aut(X)$ on $H^2(X,\ZZ)$, 
the following two groups are defined 
\begin{eqnarray*}
\Aut_s(X)
&:=&\left\{f\in\Aut(X)~\middle|~ H^2(f)|_{H^{2,0}(X)}=\mathrm{id}_{H^{2,0}(X)} \right\}\\
&=&\left\{f\in\Aut(X)~\middle|~ H^2(f)|_{T(X)}=\mathrm{id}_{T(X)} \right\}\\
\Aut_t(X)&:=&
\left\{f\in\Aut(X)~\middle|~ H^2(f)|_{{\rm NS}(X)}=\mathrm{id}_{{\rm NS}(X)} \right\}. 
\end{eqnarray*}
These two subgroups are normal. 
The group $\Aut_s(X)\subset\Aut(X)$ is of finite index, 
and its element is called a {\it symplectic automorphism}. 
The natural group homomorphisms 
\[
\Aut_s(X)\rightarrow \mathrm{O}(\mathrm{NS}(X))
~~~~~\text{  and }~~~~~
\Aut_t(X)\rightarrow \mathrm{O}(T(X)).
\]
are injective, thus $\Aut_t(X)$ is also a finite cyclic group. 

For any autoequivalence $\Phi_\E\in\Aut(\D^b(X))$, 
we define the {\it cohomological Fourier--Mukai transform} $\Phi^H_{\E}:H^{*}(X,\ZZ) \iso H^{*}(X,\ZZ)$ associated to $\Phi_\E$ by
\[
\Phi^H_\E(v):=p_*(q^*(v)\cdot v(\E)), 
\]
which is a Hodge isometry of $\widetilde{H}(X,\ZZ)$, 
thus induces the action 
\[
\Aut(\D^b(X))\rightarrow \mathrm{O}(\widetilde{H}(X,\ZZ)). 
\]
The two subgroups $\Aut_0(\D^b(X))$ and $\Aut_{{\rm CY}}(\D^b(X))$ of $\Aut(\D^b(X))$ are defined by
\begin{eqnarray*}
\Aut_0(\D^b(X))&:=&
\left\{\Phi\in\Aut(\D^b(X))~\middle|~
\Phi^H=\mathrm{id}_{H^*(X,\ZZ)} \right\}\\
\Aut_{{\rm CY}}(\D^b(X))
&:=&\left\{\Phi\in\Aut(\D^b(X))~\middle|~ \Phi^H|_{H^{2,0}(X)}=\mathrm{id}_{H^{2,0}(X)} \right\}\\
&=&\left\{\Phi \in \Aut(\D^b(X)) ~\middle|~ \Phi^H|_{T(X)}=\mathrm{id}_{T(X)} \right\}. 
\end{eqnarray*}
These two subgroups are normal, and clearly $\Aut_0(\D^b(X))\subset\Aut_{{\rm CY}}(\D^b(X))$. 
The group $\Aut_{{\rm CY}}(\D^b(X))\subset\Aut(\D^b(X))$ is of finite index, 
and its element is called a {\it Calabi--Yau autoequivalence}.
We note that $\Phi\in\Aut(\D^b(X))$ is Calabi--Yau if and only if 
it respects the Serre duality pairings
\[
\Hom(E,F[i])\times\Hom(F,E[2-i])\to\CC
\]
induced by a choice of holomorphic volume forms in $H^{2,0}(X)$, see \cite[Appendix A]{BB}. 

We recall the definition of centralizer and center groups. 
\begin{dfn}\label{def-center}
Let $G$ be a group, and $H$ a subgroup of $G$. 
\begin{enumerate}
\item
The {\it centralizer group} $C_G(H)$ of $H$ in $G$ is defined by
\[
C_G(H):=\left\{g\in G \mid gh=hg\text{ for all }h\in H\right\}.
\]
\item
The {\it center group} $Z(G)$ of $G$ is defined by
\[
Z(G):=C_G(G). 
\]
\end{enumerate}
\end{dfn}
It is easy to see $\Aut_t(X)\subset Z(\Aut(X))$, 
but these are not equal in general. 

\subsection{Stability conditions on $\D^b(X)$}


We review the space of Bridgeland stability conditions on the derived categories of K3 surfaces and the action on it of the autoequivalence group. 

Let $X$ be a K3 surface and 
fix a norm $||\cdot||$ on $\N(\D^b(X))\otimes \RR$. 
\begin{dfn}[{\cite[Definition 5.1]{Bri1}}]
A {\rm (numerical) stability condition} $\sigma=(Z,\P)$ on $\D^b(X)$ consists of a group homomorphism
$Z: \N(\D^b(X))\to\CC$ called {\it central charge} and a family $\P=\{\P(\phi)\}_{\phi\in\RR}$ of full additive subcategory of $\D^b(X)$ called {\it slicing}, such that
\begin{enumerate}
\item
For $0\neq E\in\P(\phi)$, we have $Z(v(E))=m(E)\exp(i\pi\phi)$ for some $m(E)\in\RR_{>0}$. 
\item
For all $\phi\in\RR$, we have $\P(\phi+1)=\P(\phi)[1]$.
\item
For $\phi_1>\phi_2$ and $E_i\in\P(\phi_i)$, we have $\Hom(E_1,E_2)= 0$.
\item
For each $0\neq E\in\D^b(X)$, there is a collection of exact triangles called {\it Harder--Narasimhan filtration} of $E$: 
\begin{equation}\label{HN}
\begin{xy}
(0,5) *{0=E_0}="0", (20,5)*{E_{1}}="1", (30,5)*{\dots}, (40,5)*{E_{p-1}}="k-1", (60,5)*{E_p=E}="k",
(10,-5)*{A_{1}}="n1", (30,-5)*{\dots}, (50,-5)*{A_{p}}="nk",
\ar "0"; "1"
\ar "1"; "n1"
\ar@{.>} "n1";"0"
\ar "k-1"; "k" 
\ar "k"; "nk"
\ar@{.>} "nk";"k-1"
\end{xy}
\end{equation}
with $A_i\in\P(\phi_i)$ and $\phi_1>\phi_2>\cdots>\phi_p$. 
\item
(support property)
There exists a constant $C>0$ such that for all $0\neq E\in\P(\phi)$, 
we have 
\begin{equation}
||E||<C|Z(E)|.
\end{equation}
\end{enumerate}
\end{dfn}
For any interval $I\subset\RR$, define $\P(I)$ to be the extension-closed subcategory of $\D^b(X)$ generated by the subcategories $\P(\phi)$ for $\phi\in I$. 
Then $\P((0,1])$ is the heart of a bounded t-structure on $\D^b(X)$, hence an abelian category.  
The full subcategory $\P(\phi)\subset\D^b(X)$ is also shown to be abelian. 
A non-zero object $E\in\P(\phi)$ is called {\it $\sigma$-semistable} of {\it phase} $\phi_\sigma(E):=\phi$, and especially a simple object in $\P(\phi)$ is called {\it $\sigma$-stable}. 
Taking the Harder--Narasimhan filtration (\ref{HN}) of $E$, we define $\phi^+_\sigma(E):=\phi_\sigma(A_1)$ and $\phi^-_\sigma(E):=\phi_\sigma(A_p)$. 
The object $A_i$ is called {\it $\sigma$-semistable factor} of $E$. 
Define ${\rm Stab}(X)$ to be the set of numerical stability conditions on $\D^b(X)$. 

We prepare some terminologies on the stability on the heart of a $t$-structure on $\D^b(X)$. 
\begin{dfn}
Let $\A$ be the heart of a bounded $t$-structure on $\D^b(X)$. 
A {\it stability function} on $\A$ is a group homomorphism $Z: \N(\D^b(X))\to\CC$ such that for all $0\neq E\in\A\subset\D^b(X)$, the complex number $Z(E)$ lies in the semiclosed upper half plane ${\mathbb H}_{-}:=\{re^{i\pi\phi}\in\CC~|~r\in\RR_{>0},\phi\in(0,1]\}\subset\CC$. 
\end{dfn}
Given a stability function $Z: \N(\D^b(X))\to\CC$ on $\A$, the {\it phase} of an object $0\neq E\in\A$ is defined to be $\phi(E):=\frac{1}{\pi}{\rm arg}Z(E)\in(0,1]$. 
An object $0\neq E\in\A$ is {\it $Z$-semistable} (resp. {\it $Z$-stable}) if for all subobjects $0\neq A\subset E$, we have $\phi(A)\le\phi(E)$ (resp. $\phi(A)<\phi(E)$). 
We say that a stability function $Z$ satisfies {\it the Harder--Narasimhan property} if 
each object $0\neq E\in\A$ admits a filtration (called Harder--Narasimhan filtration of $E$) 
$0=E_0\subset E_1\subset E_2\subset\cdots\subset E_m=E$ such that $E_i/E_{i-1}$ is $Z$-semistable for $i=1,\cdots,m$ with $\phi(E_1/E_0)>\phi(E_2/E_1)>\cdots>\phi(E_m/E_{m-1})$, 
and {\it the support property} if there exists a constant $C>0$ such that for all $Z$-semistable objects $E\in\A$, 
we have $||E||<C|Z(E)|$. 

The following proposition shows the relationship between stability conditions and stability functions on the heart of a bounded $t$-structure. 
\begin{prop}[{\cite[Proposition 5.3]{Bri1}}]
To give a stability condition on $\D^b(X)$ is equivalent to giving the heart $\A$ of a bounded t-structure on $\D^b(X)$, and a stability function $Z$ on $\A$ with the Harder--Narasimhan property and the support property. 
\end{prop}
For the proof, we construct the slicing $\P$, from the pair $(Z,\A)$, by 
\[
\P(\phi):=\{E\in\A~|~E\text{ is }Z\text{-semistable with }\phi(E)=\phi\}\text{ for }\phi\in(0,1],
\]
and extend for all $\phi\in\RR$ by $\P(\phi+1):=\P(\phi)[1]$. 
Conversely, for a stability condition $\sigma=(Z,\P)$, the heart $\A$ is given by $\A:=\P_\sigma((0,1])$. 
We also denote stability conditions by $(Z,\A)$.

We recall that the Mukai vector $v:\N(\D^b(X)) \iso \NN(X)$ and the Mukai pairing $(-,-)$ on $\NN(X)\subset\widetilde{H}(X,\ZZ)$, 
then the central charge of a numerical stability condition takes the form $Z(-)=(\Omega, v(-))$ for some $\Omega\in\NN(X)\otimes \CC$. Bridgeland constructed a family of stability conditions on $\D^b(X)$ as follows: 
Let
\begin{eqnarray*}
V(X)&:=&\{\beta+i\omega\in {\rm NS}(X)\otimes\CC~|~\beta,\omega\in{\rm NS(X)}\otimes\RR,~\omega: \RR\text{-ample}, \\
&&({\rm exp}(\beta+i\omega), v(E))\notin\RR_{\le0}\text{ for all spherical sheaves }E\}. 
\end{eqnarray*}
For $\beta+i\omega\in V(X)$, we set $Z_{\beta,\omega}(-):=({\rm exp}(\beta+i\omega), v(-))$. 
The category $\A_{\beta,\omega}$ is the heart of a t-structure obtained by tilting the standard t-structure with respect to the torsion pair $(\T_{\beta,\omega},\F_{\beta,\omega})$ on ${\rm Coh}(X)$ given by
\begin{eqnarray*}
\T_{\beta,\omega}&:=&\{E\in{\rm Coh}(X)~|~E\text{ is a torsion sheaf or }\mu_\omega^-(E/\text{torsion part})>\beta.\omega\}\\
\F_{\beta,\omega}&:=&\{E\in{\rm Coh}(X)~|~
E\text{ is torsion free and }\mu_\omega^+(E)\le\beta.\omega\}, 
\end{eqnarray*}
where for a torsion free sheaf $E$, 
$\mu_\omega^+(E)$ (resp. $\mu_\omega^-(E)$) is the maximal (resp. minimal) slope of $\mu_\omega$-semistable factors of $E$.  
Then $(Z_{\beta,\omega},\A_{\beta,\omega})$ is a stability condition on $\D^b(X)$ (\cite[Lemma 6.2 and Proposition 11.2]{BriK3}). 

Let $E\in\D^b(X)$ be a non-zero object of $\D^b(X)$ and $\sigma\in {\rm Stab}(X)$ be a stability condition on $\D^b(X)$. 
The {\it mass} $m_{\sigma}(E)\in\RR_{>0}$ of $E$ is defined by
\begin{equation*}
m_{\sigma}(E):=\displaystyle\sum_{i=1}^p|Z_\sigma(A_i)|, 
\end{equation*}
where $A_1,\cdots,A_p$ are $\sigma$-semistable factors of $E$. 
The following generalized metric (\emph{i.e.} with values in $[0,\infty]$) $d_B$ on ${\rm Stab}(X)$ is defined by Bridgeland (\cite[Proposition 8.1]{Bri1}): 
\begin{equation*}
d_B(\sigma, \tau):=\sup_{E\neq0}\left\{|\phi^+_\sigma(E)-\phi^+_\tau(E)|, |\phi^-_\sigma(E)-\phi^-_\tau(E)|, \middle|\log\frac{m_\sigma(E)}{m_\tau(E)}\middle|\right\}\in[0,\infty]. 
\end{equation*}
This generalized metric induces the topology on ${\rm Stab}(X)$. 
Then the generalized metric $d_B$ takes a finite value on each connected component ${\rm Stab}^\circ(X)$ of ${\rm Stab}(X)$, 
thus $({\rm Stab}^\circ(X),d_B)$ is a metric space in the strict sense. 

\begin{thm}[{\cite[Theorem 7.1]{Bri1}}]
The map
\begin{equation}
{\rm Stab}(X)\to\NN(X)\otimes\CC
;~\sigma=((\Omega, v(-)),\P)\mapsto \Omega
\end{equation}
is a local homeomorphism, where $\NN(X)\otimes\CC$ is equipped with the natural linear topology. 
\end{thm}
Therefore the space ${\rm Stab}(X)$ (and each connected component ${\rm Stab}^\circ(X)$) naturally admits a structure of finite dimensional complex manifolds. 

There is a left action of $\Aut(\D^b(X))$ on ${\rm Stab}(X)$ given by
\begin{equation}
F.\sigma:=(Z_\sigma(F^{-1}(-)), \{F(\P_\sigma(\phi))\})~\text{ for }\sigma\in {\rm Stab}(X),~F\in\Aut(\D^b(X)). 
\end{equation}
This action of $\Aut(\D^b(X))$ is isometric with respect to $d_B$. 


Let ${\rm Stab}^{\dagger}(X)$ be the connected component of ${\rm Stab}(X)$ containing the set of geometric stability conditions \emph{i.e.} one for which all structure sheaves of points are stable of the same phase. 
It is easy to check that the above stability condition $(Z_{\beta,\omega}, \A_{\beta,\omega})$ for each $\beta+i\omega\in V(X)$ is geometric. 
\begin{dfn}\label{dist-pres-auteq}
The group $\Aut^{\dagger}(\D^b(X))$ is defined as the subgroup of $\Aut(\D^b(X))$ which preserves the connected component ${\rm Stab}^{\dagger}(X)$. 
\end{dfn}
\begin{prop}[{\cite[proof of Proposition 7.9]{Har}}]
\label{Hartmann}
The following autoequivalences preserve the distinguished component ${\rm Stab}^{\dagger}(X)$ 
\emph{i.e.} are elements in $\Aut^{\dagger}(\D^b(X))${\rm :} 
\begin{itemize}
\item
shift $[n]$ for $n\in\ZZ$
\item
line bundle tensor $-\otimes\mathcal{L}$ for $\mathcal{L}\in{\rm Pic}(X)$
\item
pullback $f^*$ for $f\in\Aut(X)$
\item
The composition
\[
g^*\circ\Phi_\E^{-1}:~\D^b(X)\to\D^b(M)\to\D^b(X), 
\]
where 
$M$ is a $2$-dimensional fine compact moduli space of Gieseker-stable torsion free sheaves on $X$
with the universal family $\E$ such that $g: X\xrightarrow{\sim} M$.  
\item
spherical twist $T_A$ along a Gieseker-stable spherical bundle $A$ (e.g. $\mathcal{O}_X$). 
\item
spherical twist $T_{\mathcal{O}_C}$ along $\mathcal{O}_C$ for any $(-2)$-curve $C$ on $X$
\end{itemize}
\end{prop}
We additionally define the following groups
\begin{eqnarray*}
\Aut^{\dagger}_{{\rm CY}}(\D^b(X))&:=&\Aut_{{\rm CY}}(\D^b(X))\cap\Aut^{\dagger}(\D^b(X))\\
\Aut^{\dagger}_0(\D^b(X))&:=&\Aut_0(\D^b(X))\cap\Aut^{\dagger}(\D^b(X))
\end{eqnarray*}

Define the open subset $\P(X)\subset \NN(X)\otimes\CC$
consisting of vectors $\Omega\in\NN(X)\otimes\CC$ whose real and imaginary parts span a positive definite 2-plane $(\langle{\rm Re}\Omega,{\rm Im}\Omega \rangle_\RR, (-,-))$ in $\NN(X)\otimes\RR$. 
This subset has two connected components, distinguished by the orientation induced on this 2-plane; let $\P^+(X)$ be the component containing vectors of the form $(1,i\omega,-\frac{1}{2}\omega^2)$ for an ample class $\omega\in{\rm NS}(X)\otimes\RR$. Consider the root system
\[
\Delta(X):=\{\delta\in\NN(X)~|~(\delta,\delta)=-2\}
\]
consisting of $(-2)$-classes in $\NN(X)$, and the corresponding hyperplane complement
\[
\P_0^+(X):=\P^+(X)\backslash \bigcup_{\delta\in\Delta(X)}\delta^{\bot}.
\]

\begin{thm}[{\cite[Theorem 1.1]{BriK3}}]\label{aut-deck}
The map
\[
\pi: {\rm Stab}^{\dagger}(X) \to \P_0^+(X);~
((\Omega, v(-)), \P)\mapsto \Omega
\]
is a normal covering, and the group $\Aut^{\dagger}_0(\D^b(X))$
is identified with the group of deck transformations of $\pi$. 
\end{thm}

The following is a Bridgeland conjecture on the space of stability conditions on $\D^b(X)$ and the action on it of $\Aut(\D^b(X))$. 
\begin{conj}[{\cite[Conjecture 1.2]{BriK3}}]\label{Bri-conjecture}
Let $X$ be a K3 surface. 
Then 
\begin{enumerate}
\item
${\rm Stab}^{\dagger}(X)$ is simply-connected. 
\item
any autoequivalence of $\D^b(X)$ preserves the distinguished component ${\rm Stab}^{\dagger}(X)$, equivalently we have 
\[
\Aut^{\dagger}(\D^b(X))=\Aut(\D^b(X)). 
\]
\end{enumerate}
\end{conj}
This conjecture clearly implies an isomorphism
\[
\pi_1(\P_0^+(X))\simeq \Aut_0(\D^b(X)).
\]
\begin{thm}[{\cite[Theorem 1.3]{BB}}]\label{BB-thm}
Conjecture \ref{Bri-conjecture} holds in the case of Picard rank one.
\end{thm}
More strongly, Bayer--Bridgeland proved the contractibility of ${\rm Stab}^{\dagger}(X)$ in \cite{BB}. 

In Section \ref{section-center}, we consider the center groups $Z(\Aut^{\dagger}(\D^b(X)))$, $Z(\Aut^{\dagger}_{{\rm CY}}(\D^b(X)))$, and the centralizer group $C_{\Aut^{\dagger}(\D^b(X))}(\Aut^{\dagger}_{{\rm CY}}(\D^b(X)))$. 

\section{The center of autoequivalence groups of K3 surfaces}\label{section-center}
Let $X$ be a K3 surface of any Picard rank.  
\subsection{Main results}
We compute the center groups of $\Aut^{\dagger}(\D^b(X))$ and $\Aut_{{\rm CY}}^{\dagger}(\D^b(X))$ and the centralizer groups $C_{\Aut^{\dagger}(\D^b(X))}(\Aut_{{\rm CY}}^{\dagger}(\D^b(X)))$.
\begin{thm}\label{main}
Let $X$ be a K3 surface, $m_X$ the  order of $\Aut_t(X)$, and $f_t$ a generator of $\Aut_t(X)$.
Then we have the following
\begin{enumerate}
\item
$Z(\Aut^{\dagger}(\D^b(X)))
=C_{\Aut^{\dagger}(\D^b(X))}(\Aut^{\dagger}_{{\rm CY}}(\D^b(X)))
=\Aut_t(X)\times\ZZ[1]
\simeq(\ZZ/m_X)\times\ZZ$. 
\item
\begin{eqnarray*}
Z(\Aut^{\dagger}_{{\rm CY}}(\D^b(X)))
&=&
\begin{cases}
\langle ~(f_t^*)^{m_X/2}\circ[1]~\rangle & \mbox{if }m_X\mbox{ is even}\\
\ZZ[2] & \mbox{if }m_X\mbox{ is odd}
\end{cases}\\
&\simeq&\ZZ. 
\end{eqnarray*}
\end{enumerate}
\end{thm}
\begin{pf}
\begin{enumerate}
\item
We note that 
$\Aut(X)\ltimes\Pic(X)\subset\Aut^{\dagger}(\D^b(X))$ and 
$T_{\mathcal{O}_X}\in\Aut^{\dagger}(\D^b(X))$, 
see Proposition \ref{Hartmann}. 
Fix any $\Phi\in Z(\Aut^{\dagger}(\D^b(X)))$. 
By Theorem \ref{commute-sph}, the relation $\Phi\circ T_{\mathcal{O}_X}=T_{\mathcal{O}_X}\circ\Phi$ 
implies $\Phi(\mathcal{O}_X)=\mathcal{O}_X[i]$ for some $i\in\ZZ$. 
For any line bundle $\mathcal{L}\in\Pic(X)$ on $X$, 
we have $(\Phi\circ[-i])(\mathcal{L})=\mathcal{L}$ by $\Phi\circ(-\otimes\mathcal{L})=(-\otimes\mathcal{L})\circ\Phi$. 
By similar arguments as in the proof of \cite[Lemma A.2]{HuySph},
$\Phi\circ[-i]$ is in $\Aut_t(X)$, hence 
$\Phi\in\langle \Aut_t(X), [1]\rangle\simeq\Aut_t(X)\times \ZZ[1]$. 
It remains to show that 
$\Aut_t(X)\subset Z(\Aut^{\dagger}(\D^b(X)))$. 
Fix any $\Psi\in\Aut^{\dagger}(\D^b(X))$ and set
\[
\Psi_t:=\Psi\circ f_t^*\circ\Psi^{-1}\circ(f_t^*)^{-1}
\in\Aut^{\dagger}(\D^b(X)).
\]
Since the transcendental part of the Hodge structure $\widetilde{H}(X,\ZZ)$ is equal to $T(X)$, 
$\NN(X)\oplus T(X)\subset\widetilde{H}(X,\ZZ)$ is of finite index, 
so that $\Psi_t\in\Aut^{\dagger}_0(\D^b(X))$. 
By \cite[Lemma A.3]{HuySph}, 
$f_t^*$ acts on ${\rm Stab}^{\dagger}(X)$ trivially,
hence $\Psi_t$ also does.  
We thus have $\Psi_t={\rm id}_{\D^b(X)}$ since
$\Aut^{\dagger}_0(\D^b(X))$ is isomorphic to the group of deck transformations of the normal cover 
${\rm Stab}^{\dagger}(X)\to\P^+_0(X)$, see Theorem \ref{aut-deck}. 
We therefore have
$f_t^*\in Z(\Aut^{\dagger}(\D^b(X)))$, 
thus $Z(\Aut^{\dagger}(\D^b(X)))=\Aut_t(X)\times\ZZ[1]$. 

By $T_{\mathcal{O}_X},~-\otimes\mathcal{L}\in\Aut^{\dagger}_{{\rm CY}}(\D^b(X))$, 
similar arguments give 
\[
C_{\Aut^{\dagger}(\D^b(X))}(\Aut^{\dagger}_{{\rm CY}}(\D^b(X)))\subset\Aut_t(X)\times \ZZ[1]. 
\]
The other inclusion follows from
\[
\Aut_t(X)\times \ZZ[1]
=Z(\Aut^{\dagger}(\D^b(X)))
\subset C_{\Aut^{\dagger}(\D^b(X))}(\Aut^{\dagger}_{{\rm CY}}(\D^b(X))). 
\]
\item
By (i), we have
\begin{eqnarray*}
Z(\Aut^{\dagger}_{{\rm CY}}(\D^b(X)))
&=&C_{\Aut^{\dagger}(\D^b(X))}(\Aut^{\dagger}_{{\rm CY}}(\D^b(X)))\cap\Aut^{\dagger}_{{\rm CY}}(\D^b(X))\\
&=&(\Aut_t(X)\times \ZZ[1])\cap\Aut^{\dagger}_{{\rm CY}}(\D^b(X)). 
\end{eqnarray*}
Each element of $\Aut_t(X)\times\ZZ[1]$ is of the form $(f_t^*)^l\circ[i]$ for some $l\in\ZZ/m_X$ and $i\in\ZZ$. 
Then $((f_t^*)^l\circ[i])^H|_{H^{2,0}(X)}=(-1)^i(f_t^{*l})^H|_{H^{2,0}(X)}$ is equal to ${\rm id}_{H^{2,0}(X)}$ if and only if 
(a) $i$ is odd and $(f_t^{*l})^H|_{H^{2,0}(X)}=-{\rm id}_{H^{2,0}(X)}$,
or
(b) $i$ is even and $(f_t^{*l})^H|_{H^{2,0}(X)}={\rm id}_{H^{2,0}(X)}$ \emph{i.e.} $l=0$. 
By the faithfulness of the action of $\Aut_t(X)$ on $H^{2,0}(X)$, the case (a) is realized only when $m_X$ is even and $l=\frac{m_X}{2}$. 
We therefore have
\[
(\Aut_t(X)\times \ZZ[1])\cap\Aut^{\dagger}_{{\rm CY}}(\D^b(X))
=
\begin{cases}
\langle ~(f_t^*)^{m_X/2}\circ[1]~\rangle & \mbox{if }m_X\mbox{ is even}\\
\ZZ[2] & \mbox{if }m_X\mbox{ is odd},
\end{cases}
\]
which completes the proof. 
\end{enumerate}
\qed
\end{pf}
When 
Conjecture \ref{Bri-conjecture} (i) and (ii) are true, 
the group $\Aut^{\dagger}_{{\rm CY}}(\D^b(X))/\ZZ[2]$ is naturally isomorphic to the orbifold fundamental group of the stringy K\"ahler moduli space of $X$ (\cite[Section 7]{BB} and \cite[Conjecture 3.14]{HuyK3}), 
so this quotient group is important in the context of mirror symmetry. 
\begin{cor}\label{orb-fund}
Let $X$ be a K3 surface. 
Then we have 
\begin{eqnarray*}
Z(\Aut^{\dagger}_{{\rm CY}}(\D^b(X))/\ZZ[2])
&\simeq&Z(\Aut^{\dagger}_{{\rm CY}}(\D^b(X)))/\ZZ[2]\\
&\simeq&
\begin{cases}
\ZZ/2 & \mbox{if }m_X\mbox{ is even}\\
trivial & \mbox{if }m_X\mbox{ is odd}. 
\end{cases}
\end{eqnarray*}
\end{cor}
\begin{pf}
Let $\varphi: Z(\Aut^{\dagger}_{{\rm CY}}(\D^b(X)))\to Z(\Aut^{\dagger}_{{\rm CY}}(\D^b(X))/\ZZ[2])$ be a natural group homomorphism induced by the quotient. 
By Theorem \ref{main}(ii), 
${\rm ker}~\varphi=\ZZ[2]$ whether $m_X$ is even or not. 
We note that $\overline{\Phi}\in Z(\Aut^{\dagger}_{{\rm CY}}(\D^b(X))/\ZZ[2])$ if and only if 
any $\Psi\in \Aut^{\dagger}_{{\rm CY}}(\D^b(X))$ satisfies 
$\Psi\circ\Phi\circ\Psi^{-1}\circ\Phi^{-1}=[2l]$ for some $l\in\ZZ$. 

Let 
\[
h_{-}(-): \RR\times\Aut(\D^b(X))\to\RR;~
(t,\Psi')\mapsto h_t(\Psi')
\]
be the categorical entropy (\cite[Definition 2.5]{DHKK}). 
Then by \cite[Lemma 2.7(v) and Corollary 2.10]{KST}, we have 
\begin{eqnarray*}
h_t(\Psi\circ\Phi\circ\Psi^{-1})
&=&h_t(\Phi)\\
h_t(\Phi[2l])
&=&h_t(\Phi)+2lt
\end{eqnarray*}
for all $t\in\RR$, which concludes that $l=0$. 
The morphism $\varphi$ is therefore surjective, 
which completes the proof. 
\qed
\end{pf}
\subsection{Examples}\label{subsec-eg}
We here collect several examples of the order $m_X$ of the finite cyclic group $\Aut_t(X)$, which determines the center groups by Theorem \ref{main} and Corollary \ref{orb-fund}. 
\begin{eg}\label{odd-Pic}
Let $X$ be a K3 surface of odd Picard rank. 
Then $\mathrm{O}(T(X))$ is isomorphic to $\ZZ/2$ (cf. \cite[Cor {\bf 3}.3.5]{HuyK3}). 
Therefore $m_X=2$ (resp. $m_X=1$) if and only if $\Aut_t(X)=\mathrm{O}(T(X))$ 
(resp. $\Aut_t(X)=\{{\rm id}_X\}$). 
\end{eg}
\begin{eg}
Let $X$ be a K3 surface of Picard rank 1, and $H$ the ample generator of $\mathrm{NS}(X)$. 
As a special case of Example \ref{odd-Pic}, we have the following: 
\begin{enumerate}
\item
If $H^2=2$, then $\Aut(X)=\Aut_t(X)=\langle i^*\rangle=\mathrm{O}(T(X))$, 
where $i\in\Aut_t(X)$ is the covering involution of the double cover $X\to\PP^2$ branched along a smooth curve of degree six. 
\item
If $H^2>2$, then $\Aut(X)=\Aut_t(X)=\{\mathrm{id}_X\}\neq \mathrm{O}(T(X))$. 
\end{enumerate}
Moreover as mentioned in Theorem \ref{BB-thm}, 
we have $\Aut^{\dagger}(\D^b(X))=\Aut(\D^b(X))$. 
\end{eg}
\begin{eg}
Let $X$ be a K3 surface of Picard rank 2, with infinite automorphism group. 
Then by \cite[Corollary 1]{GLP}, $\Aut(X)$ is isomorphic to $\ZZ$ or $\ZZ/2*\ZZ/2$. 
Since $\Aut_t(X)$ is finite and $\Aut_t(X)\subset Z(\Aut(X))$, 
one has $m_X=1$. 
\end{eg}
\begin{eg}
Let $X_3$ and $X_4$ be the K3 surfaces of Picard rank 20 whose transcendental lattices are of the form 
\[
T(X_3)
=
\begin{pmatrix}
2 & 1 \\
1 & 2 \\
\end{pmatrix}
\text{  and  }
T(X_4)
=
\begin{pmatrix}
2 & 0 \\
0 & 2 \\
\end{pmatrix}
\]
respectively, see \cite{SI} and \cite[Corollary {\bf 14}.3.21]{HuyK3}. 
By Vinberg (\cite[Theorem in 2.4 and Theorem in 3.3]{Vin}), 
one has $\Aut_t(X)\simeq U_0$ in his notation, hence $m_{X_3}=3$ and $m_{X_4}=2$. 
\end{eg}
\begin{eg}
Let $X$ be a K3 surface with an unimodular or 2-elementary transcendental lattice. 
Then there exists an involution $\sigma\in\Aut(X)$ satisfying 
$H^2(\sigma)|_{{\rm NS}(X)}={\rm id}_{{\rm NS}(X)}$ and $H^2(\sigma)|_{T(X)}=-{\rm id}_{T(X)}$
, hence $m_X$ is even. 
\end{eg}
\begin{eg}
Let $X$ be a K3 surface with $\varphi(m_X)={\rm rk}~T(X)$, 
where $\varphi$ is the Euler function. 
Then $m_X$ is in the set
\[
\{12,28,36,42,44,66,3^k(1\le k\le3),5^l(l=1,2),7,11,13,17,19\}
\]
and $X$ is uniquely determined by $m_X$
due to Kondo (\cite[Main Theorem]{Kon}), Vorontsov (\cite{Vor}), Machida--Oguiso (\cite[Theorem 3]{MO}) and Oguiso--Zhang (\cite[Theorem 2]{OZ}). 
Especially, $m_X$ is even (resp. odd) if and only if $T(X)$ is unimodular (resp. non-unimodular). 
\end{eg}

\appendix

\section{Fully faithful functors and the autoequivalence group for graded dual numbers}\label{appendix}

Using Proposition \ref{class-fin-dim-mod}, we can describe the autoequivalence groups $\Aut(\D(A_d))$ and $\Aut(\D(A_d)^c)$ for the dg-algebra of the graded dual numbers $A_d$.

Recall that, for a morphism of dg-algebras $\phi : B \rightarrow B'$, the functor $\mathrm{Ind}_{\phi} \colon \D(B) \rightarrow \D(B')$ is defined as tensor product with $B$ seen as an $B\text{-}B'$ bimodule. 
For $\lambda \in \CC^{\times}$ let $\phi_{\lambda} : A_d \rightarrow A_d$ be the morphism of dg-algebras defined by sending $\varepsilon$ to $\lambda \varepsilon$.

The following theorem is a generalisation of \cite[Corollay 5.2, 5.11]{AM} to the graded setting.

\begin{thm}
  Every fully faithful endofunctor of $\D(A_d)^c$ is an autoequivalence,
  and for every $\Phi \in \Aut(\D(A_d))$
  there exists $m \in \ZZ$ and $\lambda \in \CC^{\times}$ such that $\Phi \simeq \mathrm{Ind}_{\phi_{\lambda}} \circ [m]$.
\end{thm}

\begin{proof}
  If $\Phi \colon \D(A_d)^c \rightarrow \D(A_d)^c$ is fully faithful, then it must induce an isomorphism of graded algebras $\Phi \colon A_d \rightarrow \Hom^{\bullet}_{A_d}(\Phi(A_d), \Phi(A_d)) = \oplus_{n \in \mathbb{Z}} \Hom_{\D(A_d)}(\Phi(A_d), \Phi(A_d)[n])[-n]$.
  As $\Phi(A_d)$ is compact, by Proposition \ref{class-fin-dim-mod} we know that it decomposes as a direct sum of shifts of $B_n$'s for $n > 0$.
  It is easy to check that such a direct sum has endomorphism algebra equal to $A_d$ if and only if there is just one summand and it is $B_1[m]$ for some $m \in \ZZ$.
  Hence, $\Phi(A_d) \simeq A_d[m]$.
  Now notice that the only maps of graded algebras $A_d \rightarrow A_d$ (which coincide with the only dg-algebra maps from $A_d$ to itself) are the $\phi_{\lambda}$'s.
  Hence, $\Phi \circ [-m] \circ \mathrm{Ind}_{\phi_{\lambda}}^{-1}$ is an endofunctor of $\D(A_d)^c$ that acts as the identity of $A_d$.
  Given Proposition \ref{class-fin-dim-mod}, this implies that $\Phi \circ [-m] \circ \mathrm{Ind}_{\phi_{\lambda}}^{-1} \simeq \mathrm{id}$, and therefore $\Phi \simeq \mathrm{Ind}_{\phi_{\lambda}} \circ [m]$, which is an autoequivalence, and the first claim follows.

  For the second claim notice that if $\Phi \in \Aut(\D(A_d))$ then it induces an autoequivalence of $\D(A_d)^c$, and we can repeat the argument above to conclude that $\Phi \simeq \mathrm{Ind}_{\phi_{\lambda}} \circ [m]$ for some $m \in \mathbb{Z}$.
\end{proof}


\bibliography{math}
\bibliographystyle{alpha}


\end{document}